\documentclass[11pt]{amsart}

\usepackage{amsmath, latexsym, amsfonts, amssymb,amsthm,MnSymbol,stmaryrd}
\usepackage{graphicx, xcolor,hyperref,dsfont,epsfig,caption,wrapfig,subfig,pifont,tikz}
\usepackage{datetime}
\usepackage{movie15}
\usepackage[mathscr]{eucal}
\hypersetup{
    linktoc=page,
    linkcolor=red,          
    citecolor=blue,        
    filecolor=blue,      
    urlcolor=cyan,
    colorlinks=true           
}

\newcommand{\nocontentsline}[3]{}
\newcommand{\tocless}[2]{\bgroup\let\addcontentsline=\nocontentsline#1{#2}\egroup}
\numberwithin{equation}{section}

\newtheorem{theorem}{Theorem}[section]
\newtheorem{lemma}[theorem]{Lemma}
\newtheorem{proposition}[theorem]{Proposition}
\newtheorem{corollary}[theorem]{Corollary}

\newtheorem{remark}[theorem]{Remark}
\newtheorem{definition}[theorem]{Definition}

\theoremstyle{definition}

\renewcommand{\tilde}{\widetilde}          
\DeclareMathSymbol{\leqslant}{\mathalpha}{AMSa}{"36} 
\DeclareMathSymbol{\geqslant}{\mathalpha}{AMSa}{"3E} 
\DeclareMathSymbol{\eset}{\mathalpha}{AMSb}{"3F}     
\renewcommand{\leq}{\;\leqslant\;}                   
\renewcommand{\geq}{\;\geqslant\;}                   

\newcommand{\C}{\mathbb{C}}
\newcommand{\A}{\mathbb{A}}
\renewcommand{\H}{\mathbb{H}}
\newcommand{\D}{\mathbb{D}}
\newcommand{\R}{\mathbb{R}}
\newcommand{\Z}{\mathbb{Z}}
\newcommand{\N}{\mathbb{N}}

\newcommand{\cjd}{\rangle}
\newcommand{\cjg}{\langle}

\def\l{\mathbf{l}}
\def\k{\mathbf{k}}

\def\X{\mathbf{X}}
\def\Y{\mathbf{Y}}
\def\H{\mathbf{H}}
\def\bP{\mathbf{P}}
\def\bR{\mathbf{R}}
\def\P{\mathbb{P}}
\def\E{\mathbb{E}}

\newcommand{\pl}{\partial}
\newcommand{\bbar}{\overline}
\newcommand{\mc}{\mathcal}
\newcommand{\la}{\lambda}

\usepackage{xparse}

\def\eps{\varepsilon}

\def\T{\mathbb{T}}

\def\bi{\begin{itemize}}
\def\ei{\end{itemize}}

\def\bnum{\begin{enumerate}}
\def\enum{\end{enumerate}}
\def\<#1{\langle #1 \rangle}

\newcommand{\caD}{{\mathcal D}}

\def\cF{\mathcal{F}}

\usepackage{mathtools}  
\mathtoolsset{showonlyrefs} 

\author{Colin Guillarmou}
\address{Universit\'e Paris-Saclay, CNRS,  Laboratoire de math\'ematiques d'Orsay, 91405, Orsay, France.}
\email{colin.guillarmou@universite-paris-saclay.fr}

\author{Trishen S. Gunaratnam}
\address{TIFR, Mumbai, Maharashtra 400005, India.}
\email{trishen@math.tifr.res.in}
\address{ICTS-TIFR, Bengaluru, Karnataka 560089, India}
\email{trishen.gunaratnam@icts.res.in}

\author{Vincent Vargas}
\address{Universit\'e de Gen\`eve, Section des Math\'ematiques et D\'epartment de Physique Th\'eorique, Rue de Conseil-G\'en\'eral 7-9, 1201, Geneva, Switzerland.}
\email{vincent.vargas@unige.ch}

\title{2d Sinh-Gordon model on the infinite cylinder}

\begin{document}
 
 \begin{abstract}
For  $R>0$, we give a rigorous probabilistic construction on the cylinder $\R \times (\R/ (2\pi R \Z))$  of the (massless) Sinh-Gordon model. In particular we define the $n$-point correlation functions of the model and show that these exhibit a scaling relation with respect to $R$. The construction, which relies on the massless Gaussian Free Field, is based on the spectral analysis of a quantum operator associated to the model. Using the theory of Gaussian multiplicative chaos, we prove that this operator has discrete spectrum and a strictly positive ground state. 
 \end{abstract}
 
 \maketitle

\normalsize



\section{Introduction}

Over the past 15 years, there has been huge progress in the rigorous understanding of two-dimensional conformal field theories (CFTs) defined formally via a path integral with exponential interactions. An important class of these models are non-affine Toda theories associated to complex, semisimple Lie algebras $\mathfrak g$. These models arise as the quantization of the classical Toda field theories with potentials of the form
\begin{equation}
\sum_{i=1}^n e^{\gamma x \cdot e_i }, \qquad x \in \mathbb R^n,
\end{equation}
where $\gamma \in \mathbb R$ (or, more generally, $\mathbb C$) is a coupling constant, $(e_i)_{i=1,\dots,n}$ and $\cdot$ are a special choice of basis vectors and scalar product on a $n$-dimensional vector space, related to the representation theory of $\mathfrak g$ (in particular, their Cartan subalgebras). They are well-studied in the physics literature thanks to the emergence of higher order symmetry, so-called $W$-algebras, which contain the usual Virasoro algebra of CFT. The case $\mathfrak g = \mathfrak{sl}_2$ corresponds to the simplest such theory: the celebrated Liouville CFT introduced by Polyakov \cite{Polyakov81}. The probabilistic construction  of the path integral formulation of these CFTs when $\gamma \in (0,2)$ has been developed thoroughly (see \cite{DKRV16, GRV, CRV}). In the case of Liouville, a series of works \cite{KRV_DOZZ, GKRV,GKRV21_Segal} has shown the equivalence between the probabilistic construction and the bootstrap formalism of physics \cite{DornOtto94,Zamolodchikov96}. The case of $\gamma \in i(0,\sqrt{2})$ has been also recently developed \cite{GKR} in the compactified version.

On the other hand, in the case of affine Toda theories, much less is known on both mathematical and physical fronts. Affine models correspond to quantizations of the classical field theories with potentials of the form
\begin{equation}
\sum_{i=1}^n (e^{\gamma x \cdot e_i} + e^{-\gamma x\cdot e_i}), \qquad x \in \mathbb R^n,	
\end{equation}
i.e.\ including the reflected term. In this case, the quantization leads to a quantum field theory (QFT) that is not expected to exhibit conformal invariance (i.e.\ it is not a CFT), but their significance in physics stems from the expectation that they remain \emph{integrable} in an appropriate sense. The simplest QFT in this class is the Sinh-Gordon model and corresponds to the choices $\gamma \in \mathbb R$ and $\mathfrak g = \mathfrak{sl}_2$. Amongst its most interesting features are the existence of an isolated first eigenvalue in its spectrum, leading to a mass gap and exponential decay of correlations -- this is in stark contrast to the purely continuous spectrum of Liouville. Furthermore, the study of its finer properties is an active area of research and controversy in physics. We refer to \cite{Mussardo21}, and references therein, for a more in-depth survey of physics results and research directions on the Sinh-Gordon model. The Sinh-Gordon model is closely related to the Sine-Gordon model, which corresponds to choosing $\gamma = i\beta, \beta \in \mathbb R$. The Sine-Gordon model is ubiquitous in statistical physics. For example, it arises in the Coleman correspondence (see \cite{Bauerschmidt}). Let us finally also mention that affine Toda models associated to exceptional Lie algebras ($E_8$, $E_7$, etc) arise in consideration of near-critical scaling limits of planar statistical physics models, such as the Ising model, see \cite[Chapter 16]{Mus20}.

The purpose of this article is to construct the Sinh-Gordon model on an infinite cylinder\footnote{In the physics literature, this is sometimes called the finite volume Sinh-Gordon model.}, and study some fundamental properties of its spectrum and correlations. On the mathematical side, the article \cite{BGK} was developed to define the theory using a random matrix formalism (based on a conjecture of Lukyanov \cite{Luk}), but the method seems to be very hard to implement on a technical level (in comparison to the path integral approach of this paper). Also, one should mention that using stochastic quantization the paper \cite{Bara} defines the Sinh-Gordon model in an infinite cylinder, or even the full plane,  with an extra mass term compared to the construction developed here. Inclusion of a mass term generally destroys the integrability properties of the model and it is stressed in  \cite{Bara} that removing the extra mass term is an interesting and challenging problem. We give a direct construction and study the properties of the massless model. It would be very interesting to relate the approach of this paper to the approach of \cite{BGK}, but at present this seems out of reach.

We now present the main results of this paper.

\subsection{The Sinh-Gordon model}

Let $\mu > 0$ and $\gamma \in (0,2)$. Denote by $\mc{C}_R := \R \times \T_R$ the infinite cylinder of width $R>0$, with $\T_R:=\R/2\pi R\Z$ the circle of radius $R$. 
The Sinh-Gordon model with parameters $\mu$ (called the cosmological constant) and $\gamma$ on $\mc{C}_R$ is the probability measure $\langle \cdot \rangle$ formally defined via the following path integral
\begin{equation} \label{eq: formal def sinh}
\langle F \rangle=
\frac{1}{Z} \int F(\phi) e^{- \int_{\mc{C}_R}(\frac 1{4\pi} |\nabla \phi|^2 + 2\mu \cosh(\gamma \phi)) dx} D\phi,
\end{equation}
where the integral should be over some space $\Sigma$ of fields (functions or distributions) on the cylinder $\mc{C}_R$ equipped with a measure $D\phi$ that represents the ``uniform measure'',
 $F$ is a bounded and measurable function on the space $\Sigma$ of fields, and $Z$ is a normalisation constant making $D\phi/Z$ a probability measure.
In this context, the main objects of interest are the so-called vertex correlations associated to the Sinh-Gordon model, which are formally defined via the formula
\begin{equation}\label{correlationsintro}
\cjg \prod_{j=1}^m V_{\alpha_j}(z_j)\cjd
=	
\frac{1}{Z} \int \prod_{j=1}^m e^{\alpha_j \phi(z_j)} e^{- \int_{\mc{C}_R} ( \frac 1{4\pi} |\nabla \phi|^2 +2\mu \cosh(\gamma \phi) ) dx} D\phi,
\end{equation}
where above $z_1,\dots,z_m$ is a set of $m$ disjoint points in $\mc{C}_R$ and $\alpha_1,\dots,\alpha_m\in \R$ some weights. In the physics literature the points $z_j$ are called 
insertions with weights $\alpha_j$. The vertex correlations are natural observables for the theory since, as Laplace transforms of $\phi$, they encode the distribution of the field $\phi$ and we will also see that they enjoy natural scaling properties with respect to $R$. 

Since $D\phi$ does not exist mathematically, the above definition \eqref{eq: formal def sinh} is indeed formal and making sense of the path integral is not straightforward. The purpose of this article is to define \eqref{eq: formal def sinh} rigorously using probability theory. It turns out that the random field $\phi$ underlying the construction will live in a negative Sobolev space $\Sigma =H^{-s}(\mc{C}_R)$,  as is customary in the probabilistic approach to quantum field theory (also called constructive quantum field theory). Therefore, the averages of the form \eqref{eq: formal def sinh} will be defined for bounded $F: H^{-s}(\mc{C}_R) \rightarrow \R$ and the correlations \eqref{correlationsintro} will require a renormalisation procedure since the field $\phi$ will not be defined pointwise. In order to present a concrete probabilistic construction of the above formal path integrals, we must first introduce the probabilistic construction of the Hilbert space and the Hamiltonian of the theory for $R=1$. We will see that the general case can be deduced by a scaling argument. The construction is based on the \emph{Gaussian Free Field} on the circle. 

\vspace{0.1 cm}

\noindent
{\bf Convention}: In the sequel of the paper, when working in the $R=1$ case, we will omit the subscript or superscript $R$ in the notations. For instance $\mc{C}_1$ will be denoted by $\mc{C}$.

\subsection{Main result on the Sinh-Gordon Hamiltonian}

The construction of the measure and of the correlation functions is based on the theory of \emph{Gaussian multiplicative chaos} and the  spectral resolution of the Hamiltonian of the Sinh-Gordon model, that we now explain. 
The Hamiltonian is a self-adjoint operator acting on a Hilbert space $\mc{H}$, that can be obtained as a generator of a Markov semigroup. As we shall explain later, the theory on the cylinder $\mc{C}_R$ can be reduced by a scaling argument to the theory for $\mc{C}$. We thus assume $R=1$ for now, and we will write the general case in terms of the $R=1$ case.

The Hilbert space $\mc{H}$ (with scalar product $\langle \cdot , \cdot \rangle_{\mc{H}}$) is defined as 
\[ \mc{H}=L^2(H^{-s}(\T), \mu_0).\]
Above, $\T=\R/2\pi \Z$, $H^{-s}(\T)$ is the Sobolev space of order $-s<0$ on the circle $\T$, and $\mu_0$ is the law\footnote{In other words, $\mu_0$ is the pushforward by $c+\varphi$ of the infinite mass measure $dc\otimes \mathbb{P}_{\mathbb{T}}$ with $ \mathbb{P}_{\mathbb{T}}=\prod_{n \geq 1}\frac{1}{2\pi} e^{-\frac 12(x_n^2 + y_n^2)} dx_n dy_n$ on $\R\times (\R^2)^{\N^\ast}$ with $\N^\ast= \N \setminus \lbrace 0 \rbrace$ equipped with its natural Borel $\sigma$-algebra.} of the random variable $c+\varphi$ where $c$ is a constant distributed with respect to the Lebesgue measure $dc$ and
\begin{equation}
\varphi(\theta)
=
\sum_{n \neq 0} \varphi_{n} e^{ i n \theta},  \quad\theta \in \T, \quad \varphi_{n} 
:= 
\frac{x_n + i y_n}{2 \sqrt n} , \qquad \varphi_{-n} = \overline{\varphi_{n}},
\end{equation}
with $x_n,y_n$ i.i.d. Gaussians of mean $0$ and variance $1$. The random variable $\varphi$, called the Gaussian Free Field on $\T$, has covariance kernel $-\log |e^{i\theta}-e^{i\theta'}|$ and belongs to the distributions with no constant (or zero) mode $H_0^{-s}(\T):=\{u\in H^{-s}(\T)\,|\, \hat u(0) =0\}$ (where $\hat u$ denotes the Fourier transform on $\T$) almost surely for all $s>0$.

For fixed $\gamma\in (0,2)$ and $\mu>0$, we define the \emph{Sinh-Gordon Hamiltonian}  $\mathbf{H}$ as the generator of a Markov contraction 
semigroup acting on $\mc{H}$. The semigroup ${\bf T}_t=e^{-t{\mathbf H}}$ has the expression  
\begin{equation}
{\mathbf  T}_t F(c+\varphi)
=
\mathbb E_{\varphi}[F(c+B_t+\varphi_t)  e^{-\mu \sum_{\sigma = \pm 1}\int_{[0,t]\times \mathbb T} e^{\gamma \sigma c} M^\sigma_\gamma(dsd\theta)}],
\end{equation}
where $B_t$ and $\varphi_t$ are two Markov processes with values in $H^{-s}(\T)$, $\mathbb E_\varphi$ denotes expectation conditional on $\varphi$ and, for $\sigma=\pm 1$,  $M_\gamma^\sigma$ is a random measure on the cylinder $\mc{C}_1$ called Gaussian multiplicative chaos, introduced by Kahane \cite{Kahane85}.
The process $B_t+\varphi_t$ corresponds to the decomposition of the Gaussian Free Field in the 
time slices $\{t\}\times \T$ of the cylinder $\mc{C}_1$. More precisely, $B_t$
is a standard Brownian motion (with covariance $\mathbb E[B_sB_t]=\min(s,t)$) and $\varphi_t$ is an independent Gaussian process with $0$ average on $\T$, $\varphi_{t=0}=\varphi$. 
The invariant measure of the process $(B_t+\varphi_t)$ is $\mu_0$ and under the invariant measure the covariance kernel of $\varphi_t$ is  
\[\mathbb E[ \varphi_t(\theta) \varphi_{t'}(\theta')]= \log \frac{\max(e^{-t'}, e^{-t})}{|e^{-t'} e^{i \theta'}- e^{-t} e^{i \theta}|},\]
see Section \ref{sec: GFF}. The Gaussian multiplicative chaos measure is defined by renormalisation 
\[ M_\gamma^{\sigma}(dtd\theta)=\lim_{N\to \infty} 
e^{\sigma \gamma (B_t + \varphi_t^{N})-\frac{\gamma^2}{2}\mathbb E[ (\varphi_t^{N})^2]} dtd\theta, \]
where $\varphi_t^N(\theta):=\sum_{n=-N}^N\varphi_{t,n}e^{in\theta}$ is the truncated series of 
$\varphi_t(\theta)=\sum_{n\in \Z}\varphi_{t,n}e^{in\theta}$.

When acting on  the linear span $\mc{S}$ of functions of the form $F = F(x_1,y_1, \dots, x_N,y_N) \in C^\infty((\R^2)^N)$ for some $N \in \N$ and such that all its partial derivatives have at most polynomial growth at infinity, the generator of this semigroup has the form 
\begin{equation}\label{intro_H} 
{\bf H}=-\frac{\pl_c^2}{2}+\sum_{n=1}^\infty n(\pl_{x_n}^*\pl_{x_n}+\pl_{y_n}^*\pl_{y_n})+\mu (e^{\gamma c}V_+ +e^{-\gamma c}V_-),
\end{equation}
where $V_\pm$ are non-negative unbounded operators, which in the case $\gamma<\sqrt{2}$, are multiplication operators by a potential $V_\pm \in L^p(H^{-s}_0(\T))$. These potentials can be expressed in terms of a Gaussian Multiplicative Chaos on the unit circle $\T$: $V_\pm(\varphi)$ are the mass on $\T$ of  Gaussian Multiplicative Chaos measures on $\T$ associated to the random variables $\pm \varphi$ (see \eqref{Upm}). 
Although related, we warn the reader that $M_\gamma^\pm$ and $V_\pm$ are not the same Gaussian Multiplicative Chaos, as the first is $2$-dimensional while the second is $1$-dimensional.
When $\mu=0$, ${\bf H}={\bf H}^0$ becomes the \emph{Free Hamiltonian} studied in \cite[Sections 4.2 \& 4.3]{GKRV}.

In order to construct the Sinh-Gordon theory for $R= 1$, we will rely on a diagonalization result for the Hamiltonian ${\bf H}$. 
This is no restriction since
we shall show by a scaling property of Gaussian multiplicative chaos that constructing the theory for radius $R>0$ and constant $\mu>0$ 
can be reduced to constructing it for $R=1$ and coupling constant $\mu R^{\gamma Q}$ with $Q:=\frac{\gamma}{2}+\frac{2}{\gamma}$. 

We summarize the diagonalization result, along with other spectral properties of the Hamiltonian, in the following theorem (as above, we will fix $s>0$ and work on $H^{-s}(\T)$).

\begin{theorem}\label{thm:spectrum_H}
Let $\mu > 0$ and $\gamma\in(0,2)$. The Hamiltonian ${\bf H}$ generating the semigroup ${\bf T}_t$ is a self-adjoint, positive operator acting on a dense domain $\mathcal{D}(\mathbf{H}) \subset \mathcal{H}$. Furthermore, the following properties hold:
\begin{enumerate}
\item The operator $\mathbf{H}$ has discrete spectrum $(\lambda_j)_{j \geq 0}\subset \R_+$  with complete basis of normalised eigenfunctions $(\psi_j)_{j \geq 0}$,  the smallest eigenvalue $\lambda_0$ is simple and strictly positive, and there is  a unique associated normalised eigenfunction $\psi_0\in \mc{H}$, which is positive $dc\otimes d\mathbb P_\T$-almost everywhere. We call $\psi_0$ the ground state.
\item For every $j \geq 0$ and $N \geq 0$, the eigenfunctions $\psi_j$ belong to the weighted $L^p$ spaces $e^{-N|c|}L^p(H^{-s}(\T),\mu_0)$.
\item  One has the following diagonalization result
\begin{equation}\label{introdiagonalize}
e^{-t{\bf H}}f(c+\varphi)= \sum_{j=0}^{\infty} e^{-\lambda_j t}\psi_j(c+\varphi)
\langle f, \psi_j \rangle_{\mc{H}}, \qquad \forall f \in \mathcal{H}.
\end{equation}
\end{enumerate}
\end{theorem}

The proof of Theorem \ref{thm:spectrum_H} is given in Section \ref{sec:Hamiltonian}. Although some of the techniques involved are reminiscent of the construction of the Hamiltonian associated to Liouville CFT in \cite{GKRV}, the treatment of the spectral properties is somehow different
from the Liouville case, as the analogue of $\mathbf H$ does \emph{not} have discrete spectrum. Ultimately, the difference boils down to the fact that the potential  in the Liouville Hamiltonian  decays when the average of the field (the zero mode) $c$ decays to $-\infty$, whereas the Sinh-Gordon potential is confining in both direction $c\to \pm \infty$.

In the sequel, we will sometimes write $\lambda_{j}^\mu$ to stress the dependence in $\mu$ of the eigenvalues when appropriate.

\subsection{Main results on the Sinh-Gordon path integral}

Our first result gives a rigorous meaning to \eqref{eq: formal def sinh}. We shall construct the path integral \eqref{eq: formal def sinh} by considering the limit as $T\to \infty$  of the path integral on the finite cylinder $\mc{C}_{R,T}:=[-T,T]\times \T_R$:
\begin{equation} \label{eq: formal def approx sinh}
\langle F \rangle_{\mc{C}_{R,T}}
=
\frac{1}{Z_{T}} \int F(\phi) e^{- \int_{\mc{C}_{R,T}}  ( \frac {1} {4 \pi} |\nabla \phi|^2 +2\mu \cosh(\gamma \phi) ) dx} D\phi,
\end{equation}
where the formal path integral on the right-hand side will be defined probabilistically. Recall that we view the probability measure $\langle \cdot \rangle_{\mathcal C_{R,T}}$ as defined on $H^{-s}(\mathcal C_{R,T})$ for some $s>0$. Motivated by the construction of analogous measures in 1d \cite{OS99, LHB} (see also Appendix \ref{appendix: path integral}), we will in fact construct the measure on a more regular subspace, $C(\mathbb R, H^{-s}(\mathbb T_R))$, which allows us to view Sinh-Gordon fields as continuous-time stochastic processes with values in the Sobolev space $H^{-s}(\mathbb T_R)$. We will also take $R=1$ and consider the general case by a scaling argument. For $F:C([-T,T],H^{-s}(\T))\to \R$ continuous and bounded and for $R=1$ we interpret \eqref{eq: formal def approx sinh} by 
\begin{equation} \label{eq: path integral finite approx}
\langle F \rangle_{\mc{C}_{1,T}}
:=
\frac{1}{Z_T} \int_{\R}\mathbb{E} \Big[ F(c+B_{T+\bullet} + \varphi_{T+\bullet}) e^{- \mu \sum_{\sigma = \pm 1} e^{\sigma\gamma c}M_\gamma^\sigma([0,2T]\times \T)}  \Big]dc,
\end{equation}
where $T+\bullet$ denotes the shifted process (this is for convenience as it allows us to start the process from $t=0$ rather than $t=-T$) and $Z_T= \int_\mathbb R \mathbb E[e^{- \mu \sum_{\sigma = \pm 1} e^{\sigma\gamma c}M_\gamma^\sigma([0,2T]\times \mathbb T)} ]dc$ denotes the total mass of the measure. 

We will define the Sinh-Gordon model on $\mathcal{C}$ to be the probability measure obtained by taking the $T\rightarrow \infty$ limit of the measures defined by \eqref{eq: path integral finite approx}. An in-depth discussion of the precise $\sigma$-algebras involved and the reason why considering these averages defines a probability measure is contained in Section \ref{subsec: state space}.

We say that $F:C(\mathbb{R}, H^{-s}(\T_R)) \rightarrow \mathbb{R}$ is $I$-measurable for some interval $I\subset \R$ if there exists 
$F^I : C(I, H^{-s}(\T_R)) \rightarrow \mathbb{R}$ such that
\begin{align*}
F(\phi) 
= 
F^I(\phi|_I), \qquad \forall 	\phi \in C(\mathbb{R}, H^{-s}(\T_R)),
\end{align*} 
and we shall identify $F$ and $F^I$.
The following theorem establishes the existence of the Sinh-Gordon model on the infinite cylinders $\mc{C}_R$ for any $R>0$ and gives explicit formulas for expectations in terms of the stochastic process $(B_\bullet,\varphi_\bullet)$. Since our construction relies on a massless GFF, one can relate the theory on $\mc{C}_R$ to the theory on $\mc{C}$ by a simple scaling. For any $F:C(\mathbb{R}, H^{-s}(\T_R))\rightarrow \mathbb{R}$, let $F_R:C(\mathbb{R}, H^{-s}(\T)) \rightarrow \mathbb{R}$ be defined by $F_R(\phi)=F(\rho_R^*\phi)$ with $\rho_R(t,\theta)=(t/R,\theta/R)$ where $\phi$ is viewed as a distribution on $\mc{C}$.

The next theorem concerns the construction of the path integral and is proved in Section \ref{subsec: construction sinh R}.

\begin{theorem} \label{thm: sinh and one pt}
Let $\mu > 0$, $\gamma \in (0,2)$, and $R>0$. There exists a probability measure $\langle \cdot \rangle_{\mc{C}_R}$ on $C(\R,H^{-s}(\T_R))$ satisfying:\\
1) If $R=1$, for every  finite interval $I=[t_1,t_2]$ of length $|I|=t_2-t_1$ and for every $F:C(\R,H^{-s}(\T)) \rightarrow \mathbb{R}$ bounded and measurable with respect to $(B_t,\varphi_t)_{t \in [t_1,t_2]}$, the limit of $\cjg F\cjd_{\mc{C}_{1,T}}$ as $T\to \infty$ exists and 
\[ \cjg F\cjd_{\mc{C}_1} =\lim_{T\to \infty}\cjg F\cjd_{\mc{C}_{1,T}}.\]
Moreover, the expectation of $F$ is given by
\begin{align*}
\langle F \rangle_{\mc{C}_1}
&= 
e^{\lambda_{0} |I| } \int_\R \E[\psi_0(c+\varphi) \psi_0(c+B_{|I|}+\varphi_{|I|}) F\big( c+B_{\bullet -t_1}+\varphi_{\bullet-t_1}\big)  e^{-\mu  \sum_{\sigma = \pm 1} e^{\sigma\gamma c}M_\gamma^\sigma([0,|I|]\times\T)}] dc.
\end{align*} 
where $\la_0,\psi_0$ are the smallest eigenvalue and ground state of ${\bf H}$ described in Theorem \ref{thm:spectrum_H}.\\
2) If $R>0$, for every  finite interval $I=[t_1,t_2]$ and $F:C(\mathbb{R}, H^{-s}(\T_R))\to \mathbb{R}$ bounded such that $\rho_R^*F:C(\mathbb{R}, H^{-s}(\T))\to \mathbb{R}$ is measurable with respect to $(B_t,\varphi_t)_{t \in I}$,  $\langle F \rangle_{\mc{C}_{R}}$ satisfies the scaling relation
\begin{align*}
\langle F \rangle_{\mc{C}_{R}}
=
e^{\frac{\lambda_0^{\mu_R}|I|}{R} } \int_\R \E[\psi_0^{\mu_R}(c+\varphi) \psi^{\mu_R}_0(c+B_{\frac{|I|}R}+\varphi_{\frac{|I|}{R}}) F \big( (B_{\frac{\bullet-t_1}{R}}, \varphi_{\frac{\bullet-t_1}{R}})\big)  e^{-\mu_R \sum_{\sigma = \pm 1} e^{\sigma\gamma c}M_\gamma^\sigma([0,\frac{|I|}R]\times \T)}]dc,
\end{align*} 
where $Q=\frac 2\gamma + \frac \gamma 2$, $\mu_R:=\mu  R^{\gamma Q}$, $\lambda_{0}^{\mu_R}$ and $\psi^{\mu_R}_0$ denote the smallest eigenvalue and associated normalised eigenfucntion associated with the Hamiltonian $\H$ of \eqref{intro_H} with the constant $\mu$ replaced by $\mu_R$.
\end{theorem}

The proof of Theorem \ref{thm: sinh and one pt} is contained in Section \ref{sec: path integral}. It is inspired by the construction of $P(\varphi)_1$ measures in 1d (see \cite{OS99, LHB}), but there are a number of essential difficulties to overcome in order to adapt it to our setting. First of all, in the 1d case, the associated Hamiltonians of the theory are classical Schr\"odinger operators whose spectral properties are well-understood. In our case, the necessary inputs of their theory are not immediate and this is the content of Theorem \ref{thm:spectrum_H}. Secondly, in 1d the natural reference measure can be related to Wiener measure and hence Brownian motion. In two dimensions, the natural reference measure is the Gaussian Free Field, which is associated to Brownian motion and infinitely many independent Ornstein-Uhlenbeck processes. In the case of quantum electrodynamics, similar constructions of an invariant measure on paths with values in $L^2(\Omega)$ for a Gaussian probability space $\Omega$ have been obtained in \cite{Hiroshima1} using the ground state approach. In this case, the Hilbert space is $\mc{H}=L^2(\R^d)\otimes L^2(\Omega)$, the Hamiltonian generating the stochastic dynamics share similarities with our Hamiltonian ${\bf H}$ where $\mc{H}\simeq L^2(\R)\otimes L^2(\Omega)$ (with $(c,(\varphi_n)_{n\in \N})\in \R \times \Omega$), except that the potential in \cite{Hiroshima1} is a function of the variable in $\R^d$ (in the Kato class), while in our case and the potential $V$ is a function of all variables $(c,(\varphi_n)_{n\in \N})$ in $\R \times \Omega$ expressed in terms of Gaussian Multiplicative Chaos, that can not be controlled by the free field Hamiltonian ${\bf H}^0$. This complicates the analysis and require probabilistic estimates on Gaussian Multiplicative Chaos measures.

We now turn to the vertex correlations. In order to state our theorem, we will need the notion of $\gamma$-admissible insertions -- these are the analogue of the first Seiberg bounds in the context of Liouville theory \cite{DKRV16}. 
Insertions are disjoint points $z_1,\dots,z_n$ on the cylinder $\mc{C}$ which come with weights  $\alpha_1,\dots,\alpha_n\in \R$ attached to them. A set $\mc{I}$ of insertions is a finite set $\mathcal{I}=\{(\alpha_1,z_1),\dots,(\alpha_n,z_n)\}\subset \R\times \mc{C}$ such that $z_i\not= z_j$ if $i\not=j$. We say that $(\alpha_1,z_1),(\alpha_2,z_2) \in \mc{I}$ are in the same slice if there exists $t \in \R$ such that $z_1,z_2 \in \{t\} \times \T$. 
Given a set $\mathcal{I}\subset \R\times \mc{C}$ of insertions and $t \in \R$, we write $\mathcal{I}_t$ to denote the set $\{ (\alpha,z) \in \mathcal{I} \, |\, \exists \theta \in \T, z = (t,\theta) \}$. We say that a set $\mc{I}$ of insertions is $\gamma$-admissible if $|\alpha|<Q$ for all $(\alpha,z) \in \mathcal{I}$ (recall that $Q=\gamma/2+2/\gamma$). 

We will also be interested in how vertex correlations in different vertical slices of the cylinder interact. Let $\mathcal{I}$ be a set of $\gamma$-admissible insertions and let $\{t_1,\dots,t_n\}\subset \R$ (with $t_i<t_j$ for $i \not = j$) denote the finite set of slices that $\mathcal{I}$ is supported on, in the sense that
$\mathcal{I}=\sqcup_{i=1}^n \mathcal{I}_{t_i}$. We write $\mathcal{I}_{t_i}= \cup_{j=1}^{l_i} \lbrace   (\alpha_{ij},t_i, \theta_{ij}) \rbrace$ and define the function $h_i(\theta)= -\sum_{j=1}^{l_i} \log |e^{i\theta}-e^{i \theta_{ij}}|$. Then for all $i$ the variable $\psi_0(\cdot+ h_i)$ is well defined as the limit of $\psi_0(\cdot+ h_{i,\eps})$ as $\eps$ goes to $0$ where $(h_{i,\eps})_{\eps >0}$ is any smooth approximation to $h_i$: see Section \ref{subsec:_vertex_correlations} for a justification of this point.

In the following theorem, whose proof is in Section \ref{subsec:_vertex_correlations}, we establish that the vertex correlations exist for any insertion set $\mathcal{I}$ and are nontrivial provided $\mathcal{I}$ is $\gamma$-admissible. We also give explicit formulas for these correlations. Finally, we show that the two-point \emph{truncated} vertex correlation function, i.e.\ the covariance of the two vertex correlations, decays exponentially with rate proportional to the mass gap in the spectrum: $\frac{1}{R}(\la_1^{\mu_R}-\la_0^{\mu_R})$. Since the scaling in $\mu$ will be important, we denote by $\cjg \cdot\cjd_{\mc{C}_R,\mu}$ the probability measure in Theorem \ref{thm: sinh and one pt} to emphasize the $\mu$ dependence.

\begin{theorem} \label{thm: sinh correlations}
Let $\mu>0$, $\gamma\in (0,2)$, $R>0$ and fix a set $\mathcal{I}=\bigsqcup_{i=1}^n\mc{I}_{t_i}$ of insertions with $\mathcal{I}_{t_i}= \cup_{j=1}^{l_i} \lbrace   (\alpha_{ij},t_i, \theta_{ij}) \rbrace$. The vertex correlations are defined via the formula
\[\Big\cjg \prod_{(\alpha,z)\in \mc{I}}V_{\alpha}(z)\Big\cjd_{\mc{C}_R,\mu} :=\lim_{\eps\to 0}
\Big\langle \prod_{(\alpha, (t, \theta)) \in \mathcal{I}}\eps^{\alpha^2/2}e^{\alpha (c+ B_{t/R} +  \varphi^{R,\eps}(t,\theta)) } \Big\rangle_{\mc{C}_R,\mu}\, ,
\]
where the limit exists for any insertion set $\mathcal I$ and is nontrivial if and only if $\mathcal{I}$ is $\gamma$-admissible. In addition, $\gamma$-admissible vertex correlations obey the following relations:
\vspace{3mm}

\noindent 1) Let $\mathcal{I}$ be $\gamma$-admissible insertions on $\mc{C}_R$ and let $\mathcal{I}_{1/R} = \{ (\alpha, z/R) \,|\, (\alpha,z) \in \mathcal{I} \}$. Then we have the following scaling relation: 
\begin{equation}\label{scale invariance}
\Big\cjg \prod_{(\alpha,z)\in \mc{I}}V_{\alpha}(z)\Big\cjd_{\mc{C}_R,\mu} 
=
R^{\sum_{(\alpha,z) \in \mathcal{I}} \frac{\alpha^2}{2} }
\Big\cjg \prod_{(\alpha,z)\in \mc{I}_{1/R}}V_{\alpha}(z)\Big\cjd_{\mc{C},\mu_R}.
\end{equation}
with $\mu_R=R^{\gamma Q}\mu$.
 In particular, the one-point function $\langle V_\alpha(0) \rangle_{\mathcal C_R}:=\langle e^{\alpha \varphi(0)} \rangle_{\mc{C}_R,\mu}$ exists, is nontrivial for $|\alpha|<Q$, and obeys the scaling relation
\begin{equation} \label{one-pt}
\langle e^{\alpha \varphi(0)} \rangle_{\mc{C}_R,\mu}
=
R^{\alpha^2/2} \langle e^{\alpha \varphi(0)} \rangle_{\mc{C},\mu_R}.	
\end{equation}
\vspace{2mm}

\noindent 2)
One has the following formula for the correlations when $R=1$; if $t_i >0$ for all $i$, for any $t>\max(t_1,\dots,t_n)$ 
\begin{align}
\begin{split}
 \Big\cjg \prod_{(\alpha,z)\in \mc{I}}V_{\alpha}(z)\Big\cjd_{\mc{C},\mu}   & =  e^{\lambda_0 t + \frac{1}{2}\sum_{i,j}\alpha_{ij}^2t_i} \int_\R e^{\sum_{i,j}\alpha_{ij} c}\mathbb{E}\big[\psi_0\big(c+B_t+\sum_{i,j} \alpha_{ij}t_i+ \varphi_t + Ph (t,\cdot))\psi_0(c+\varphi+h) 
\\
& \qquad \qquad \times e^{-\mu \sum_{\sigma = \pm 1} e^{\sigma\gamma c} \int_{[0,t]\times \T} \prod_{i,j} |e^{-s+i\theta}-e^{-t_i+i\theta_{ij}}|^{-\gamma \sigma \alpha_{ij}} M^\sigma_\gamma(dsd\theta)}\big ]dc
\end{split}
\end{align}
where $Ph(t,\theta) = \sum_{i,j} \alpha_{ij}\mathbb E[ \varphi_t(\theta) \varphi_{t_i}(\theta_{ij})]$ is a harmonic function with boundary value $h(\theta)= \sum_{i,j} \alpha_{ij}\mathbb E[ \varphi(\theta) \varphi_{t_i}(\theta_{ij})]$ at $t=0$ and decaying to $0$ as $t\to \infty$.
If $\mathcal{I}=\{(\alpha,0)\}$ with $\alpha\in (-Q,Q)$, then  \begin{equation} 
\label{vertex explicit}
\cjg V_{\alpha}(0)\cjd_{\mc{C},\mu} = \|e^{\alpha c/2} \psi_{0}(\cdot+h)\|^2_\mc{H}
\end{equation}
with $h(\theta)=-\alpha \log|e^{i\theta}-1|$.

\vspace{3mm}

\noindent 3) For every $t \in \mathbb R$, $\theta \in \mathbb T_R$, and $\alpha \in (-Q,Q)$, we have that the corresponding two-point vertex correlations\footnote{Similar statements hold for $F,G:C(\mathbb R, H^{-s}(\mathbb T_R))\to \R$ that are bounded, measurable, and of compact support. See the proof of Theorem \ref{thm: sinh correlations} in Section \ref{subsec:_vertex_correlations}.} are positively correlated:
\begin{equation}
{\rm Cov}_{\mathcal C_R}(V_\alpha(0,\theta), V_\alpha(t,\theta)):= \langle V_\alpha(0,\theta) V_\alpha(t,\theta) \rangle_{\mathcal C_R} - \langle V_\alpha(0,\theta) \rangle_{\mathcal C_R}\langle  V_\alpha(t,\theta) \rangle_{\mathcal C_R}\geq 0. 	
\end{equation}
Furthermore, we have exponential decay of the two-point vertex correlation functions: for every $\theta_1,\theta_2 \in \mathbb T_R$ and every $\alpha_1,\alpha_2 \in (-Q,Q)$, there exists $C_{R}(\alpha_1,\theta_1; \alpha_2,\theta_2)>0$ such that, for every $|t| > 1$,
\begin{equation}
|{\rm Cov}_{\mathcal C_R}(V_{\alpha_1}(0,\theta_1), V_{\alpha_2}(t,\theta_2))| \leq e^{-(\lambda_1^{\mu_R}-\lambda_0^{\mu_R})\frac{|t|}R} C_R(\alpha_1,\theta_1; \alpha_2,\theta_2).	
\end{equation}
\end{theorem}

The proof of Theorem \ref{thm: sinh correlations} is at the end of Section \ref{sec: path integral}. The difficulty lies in the fact that the vertex correlations are singular and must be defined via regularized and renormalised approximations. Thus, Theorem \ref{thm: sinh and one pt} cannot be applied in a straightforward manner. Convergence is obtained by using probabilistic techniques and Gaussian Multiplicative Chaos estimates, together with growth bounds on the eigenfunctions of $\mathbf H$ stated in Theorem \ref{thm:spectrum_H}.

\subsection{Perspectives and open problems}
\subsubsection{On the spectrum of the theory}
In the physics literature \cite{Tesch}, there exists a conjecture by Teschner giving formulas for the full spectrum $(\lambda_j)_{j \geq 0}$ of the theory. It would be very interesting to derive these formulas. 
\subsubsection{The infinite volume limit $R \to \infty$}
The infinite volume limit corresponds to taking $R \to \infty$. By the scaling relations, this corresponds to studying the limit of the rescaled Hamiltonian 
\[ \H_{R}:=R^{-1}({\bf H}^0+\mu R^{\gamma Q}(e^{\gamma c}V_+ +e^{-\gamma c}V_-))\] 
on $\mc{H}$. In particular, it is conjectured in physics that:
\begin{itemize}
\item
In our notation the lowest eigenvalue scales as $\lambda_0^{\mu_R} \underset{R \to \infty}{\sim} c R^2$ where $c>0$ is an explicit constant and $\mu_R=\mu R^{\gamma Q}$. 
See the review \cite{Mussardo21} for instance: in the review the scaling is written as $c R$ which corresponds to the Hamiltonian $\H_R$.
\item
The infinite volume one point function defined via 
\begin{equation*}
\langle V_\alpha(0) \rangle = \underset{R \to \infty}{\lim} R^{\alpha^2/2} \langle V_\alpha (0) \rangle_{\mc{C},\mu_R}\end{equation*}
 has the following explicit form for $\alpha \in (-Q,Q)$
\begin{align*}
& \langle e^{\alpha \varphi(0) } \rangle   \\
&= \left (  -\frac{\mu \pi \Gamma (1+\frac{\gamma^2}{4})}{\Gamma (-\frac{\gamma^2}{4})}  \right )^{-\frac{\alpha^2}{2 \gamma Q}}   \exp \left (  \int_0^\infty\left ( - \frac{(e^{\frac{\alpha \gamma}{2}t }  -e^{-\frac{\alpha \gamma}{2}t }  )^2}{(e^{\frac{ \gamma^2}{4}t }  -e^{-\frac{ \gamma^2}{4}t })(e^t-e^{-t})( e^{\frac{\gamma Q}{2}t }+e^{-\frac{\gamma Q}{2}t } )}   + \frac{\alpha^2}{2}e^{-2 t}  \right )     \frac{dt}{t} \right )
\end{align*}
where recall that $Q=\frac{\gamma}{2}+\frac{2}{\gamma}$. This formula was derived by Lukyanov-Zamolodchikov \cite{LukZam}.  
\end{itemize}

\subsection*{Acknowledgements}

We thank Baptiste Cercl\'e, Romain Panis, and R\'emi Rhodes for interesting discussions. TSG was supported by the Simons Foundation, Grant 898948, HDC. The research of V.V. is supported by the SNSF grant ``2d constructive field theory  with exponential interactions".

\section{Gaussian free field and Gaussian multiplicative chaos} \label{sec: GFF}

In this section we define the two main probabilistic objects that underpin our analysis: the Gaussian free field (GFF) and Gaussian multiplicative chaos (GMC).

\subsection{Path integral and the GFF on $\mc{C}_R$}\label{s:Path_Integral}
We begin this section by explaining our choice of construction of the Sinh–Gordon model using the Gaussian Free Field.
This discussion is not fully detailed, but rather serves 
as a guideline motivating the construction presented below. 

For $s>0$, let $\mc{D}'(\mc{C}_R)$ be the space of distributions on the cylinder. 
To make sense of the path integral on the infinite-volume cylinder $\mc{C}_R$
\begin{equation}\label{path_integral_sinh}
\int_{\mc{D}'(\mathcal{C}_{R})} F(\phi) e^{- \int_{\mathcal{C}_{R}}( \frac 1{4\pi}|\nabla \phi|^2+2\mu\cosh(\gamma \phi))dx} D\phi,
\end{equation}
we will first consider a finite-volume cylinder $\mathcal{C}_{R,T} = [-T,T] \times \mathbb{T}_{R}$ and then take the limit $T\to \infty$ after renormalisation by the mass $Z{\mc{C}_{R,T}}$ (which corresponds to taking $F(\phi)=1$).
For such a bounded interval, the measure will actually live on the Sobolev space $H^{-s}(\mc{C}_{R,T})$ of negative order $-s<0$ (for any fixed $s>0$).

Since $\mathcal{C}_{R,T}$ has boundary, one possible way to interpret the path integral 
\begin{equation}\label{Path_Int_CRT}
\int_{H^{-s}(\mc{C}_{R,T})} F(\phi) e^{- \int_{\mathcal{C}_{R,T}}( \frac 1{4\pi}|\nabla \phi|^2+2\mu\cosh(\gamma \phi))dx} D\phi
\end{equation}
is to fix the value of $\phi$ at $t=\pm T$ to be $(c+\varphi,c'+\varphi')\in H^{-s}(\T_R)^2$ with $c,c'\in \R$ and $\hat{\varphi}(0)= \hat{\varphi'}(0)=0$,
and consider it as a function 
of $(c+\varphi,c'+\varphi')$. This can be done rigorously as in \cite{GKRV21_Segal} by first considering the free theory (the interaction vanishes, i.e. $\mu=0$) using the field 
\[ \phi = X_{\mc{C}_{R,T},D}+P(c+\varphi,c'+\varphi')\in H^{-s}(\mc{C}_R),\] 
where $X_{\mc{C}_{R,T},D}$ is the Dirichlet Gaussian Free Field on $\mc{C}_{R,T}$ and $P(c+\varphi,c'+\varphi')$ the harmonic function in the interior of $\mc{C}_{R,T}$ with boundary values $(c+\varphi,c'+\varphi')$ (in the weak sense). Here, 
$\varphi,\varphi'\in H^{-s}_0(\T_R)$ are centred Gaussian random variables (orthogonal to constants) and with covariance kernel 
\[ \E[ \varphi(\theta)\varphi(\theta')]=
\E[ \varphi'(\theta)\varphi'(\theta')]=-\log|e^{i\frac{\theta}{R}}-e^{i\frac{\theta'}{R}}|.\] 
The random variable  $X_{\mc{C}_{R,T},D}$ has covariance kernel $2\pi G_{\mc{C}_{R,T}}(t,\theta,t',\theta')$ where $G_{\mc{C}_{R,T}}$ is the Dirichlet Green function on $\mc{C}_{R,T}$.
We can define, for each continuous  $F:H^{-s}(\mc{C}_{R,T})\to \R^+$, a function 
\begin{equation}
\mc{A}(F)(c+\varphi, c'+\varphi')
=\mc{A}^0_{\mc{C}_{R,T}}(c+\varphi, c'+\varphi')\mathbb{E}[F(\phi)],
\end{equation}
where $\mathbb{E}$ denotes expectation with respect to the law of the Dirichlet Gaussian free field $X_{\mc{C}_{R,T},D}$ and $\mc{A}^0_{\mc{C}_{R,T}}(c+\varphi, c'+\varphi')$ is an explicit kernel that generates a Markov semigroup (with respect to the cylinder height $T$), called free propagator, on the Hilbert 
space $L^2(H^{-s}(\T_R),\mu_0^R)$ with $\mu_0^R$ the law of $c+\varphi$ (this will be discussed in Lemma \ref{Int_kernel_prop}).

The amplitude $\mc{A}(F)$ can be used to give a rigorous definition of the formal path integral of the \emph{free theory with boundary conditions}, 
\begin{equation}\label{double_conditioning} 
\int_{H_{c+\varphi,c'+\varphi'}^{-s}(\mc{C}_{R,T})} F(\phi) e^{-\frac{1}{4\pi} \int_{\mathcal{C}_{R,T}}|\nabla \phi|^2dx} D\phi,
\end{equation}
where $H_{c+\varphi,c'+\varphi'}^{-s}(\mc{C}_{R,T})$ denotes\footnote{Recall that a distribution in $H^{-s}(\mc{C}_{R,T})$ cannot a priori be restricted to a time slice, but the random variables we shall use almost surely actually do admit such restriction.} the set of $\phi\in H^{-s}(\mc{C}_{R,T})$ that admit restriction to the boundary and whose boundary 
values are $\phi|_{t=-T}=c+\varphi,\phi|_{t=T}=c'+\varphi'$.
Integrating out the variable $c'+\varphi'$ with respect to $\mu^R_0$, and provided the integral is finite, it can be checked (see the Remark below Proposition \ref{prop: sinh expectation}) using the domain Markov property of the  GFF  \cite[Theorem 1.52]{Berestycki_lqggff}  that one obtains a function 
\[ W^0_T(c+\varphi):=\int_{H^{-s}(\T_R)}\mc{A}(F)(c+\varphi, c'+\varphi')d\mu^R_0(c+\varphi')\] 
that is equal to the following expression ($\E_\varphi$ denotes expectation conditional on $\varphi$ and in the sequel $\P_\varphi$ will denote the conditional probability measure) 
\begin{equation}\label{W_T} 
W^0_T(c+\varphi)=\E_\varphi[ F(\phi^R(t)|_{t\in [-T,T]})].
\end{equation}
Here, $t\mapsto \phi^R(t)$ is the random process in $C^0([-T,T],H^{-s}(\T_R))$ given by 
$\phi^R(t)=c+X_{\mc{C}^+_{R},D}(t+T,\cdot)+P\varphi(t,\cdot)$, $P\varphi$ is the harmonic function 
on the half-cylinder $[-T,\infty)\times \T_R$ with boundary value $\varphi$ at $t=-T$ and $0$ at $t=+\infty$, and 
$X_{\mc{C}^+_{R},D}$ is the Gaussian Free Field on $\R_+ \times \T_R$ with Dirichlet condition at $t=0$ (i.e. with covariance kernel $2\pi$ times the Green function of the Dirichlet Laplacian  $\Delta_{{\mc{C}^+_{R},D}}$ on the half-cylinder $\mc{C}^+_{R}$). 

This point of view generates a natural and interesting dynamic, more specifically a Markov process, and is well adapted to the addition of the $\cosh(\gamma \phi)$ potential. In this context, i.e.\ $\mu \neq 0$, the addition of the potential induces a Markov process with corresponding semigroup acting on the \emph{same} Hilbert space as the free theory,  and with generator a self-adjoint (interacting) Hamiltonian ${\bf H}$ with discrete spectrum. Using the spectral decomposition of ${\bf H}$ and the properties of its eigenfunctions, we will be able to describe the limit of $W^0_T(c+\varphi)$ as $T\to +\infty$. The construction and spectral properties of $\H$ will be explained more precisely in Section \ref{sec:Sinh_hamiltonian}. Let us also mention that another motivation for using this Markov process approach comes from the treatment of the analogous path integrals in 1d \cite{OS99} from a statistical mechanics point of view (i.e.\ motivated by the theory of Gibbs measures). 

Let us return to the question of giving a rigorous definition of the path integral without fixed boundary conditions. The function $W^0_T$ gives a rigorous definition of the path integral 
 \begin{equation}\label{simple_conditioning}
 \int_{H_{c+\varphi}^{-s}(\mc{C}_{R,T})} F(\phi) e^{-\frac{1}{4\pi} \int_{\mathcal{C}_{R,T}}|\nabla \phi|^2dx} D\phi,
 \end{equation}
where $H_{c+\varphi}^{-s}(\mc{C}_{R,T})$ is the set of $\phi \in H^{-s}(\mc{C}_{R,T})$ that admit a restriction to $t=-T$ and whose value is $\phi|_{t=-T}=c+\varphi$.
Integrating out the  $c+\varphi$ variable with respect to $\mu_0^R$ in $W^0_T(c+\varphi)$ then gives a rigorous definition of the path integral 
 \[
 \int_{H^{-s}(\mc{C}_{R,T})} F(\phi) e^{-\frac{1}{4\pi} \int_{\mathcal{C}_{R,T}}|\nabla \phi|^2dx} D\phi.\]
We will build on this Gaussian path integral to give a rigorous definition of the path integral \eqref{Path_Int_CRT} including the nonlinear potential $\cosh(\gamma \phi)$ by means of Gaussian multiplicative chaos measures. In fact, we will even consider the same path integral, but with boundary values of $\phi$ at $t=\pm T$ to be fixed, just like in \eqref{double_conditioning} or \eqref{simple_conditioning}.
As will be explained in the Remark below Proposition \ref{prop: sinh expectation},  when $T\to \infty$, the behaviour 
 with or without fixing the boundary value at $t=\pm T$  produces the same result and allows to define \eqref{path_integral_sinh}.
 
\subsection{Markov process associated to the GFF}\label{sec:markov_process_GFF} 
Let us come back to the Gaussian path integral with the GFF. For practical purposes, one can exploit the invariance in law of the field $\phi$  under a shift  by $+2T$ in the $t$-variable to reduce the analysis on the finite volume cylinder $[0,2T]\times \T_R$ instead of $[-T,T]\times \T_R$.  We are going to build the (shifted by $2T$) Markov process $\phi^R$ explicitly from random Fourier series, allowing us to write down the explicit covariance functions of the process.

Consider $\Omega_\T := (\R^2)^{\N^*}$ endowed with its natural Borel $\sigma$-algebra $\cF_\T$, and denote by $(x_n,y_n)_{n \in \N^*}$ the canonical coordinates in this space. Let $\P_\T$ denote the probability measure on $(\Omega,\cF_\T)$ defined by
\begin{equation}\label{Gaussian_measure}
\P_\T
=
\bigotimes_{n=1}^\infty \frac{1}{2\pi} e^{-\frac 12(x_n^2 + y_n^2)} dx_n dy_n.
\end{equation}
The canonical coordinates in this probability space, $(x_n)_{n \in \N^*}$ and $(y_n)_{n \in \N^*}$, $\P_\T$-a.s. correspond to two independent sequences of i.i.d. standard normal random variables. 
Consider the random variable 
\begin{equation}\label{RVvarphi}
\varphi(\theta)
=
\sum_{n \neq 0} \varphi_{n} e^{ i n \frac{\theta}{R}},  \quad\theta \in \T_R,
\quad \textrm{where } \varphi_{n} 
:= 
\frac{x_n + i y_n}{2 \sqrt n} , \qquad \varphi_{-n} = \overline{\varphi_{n}}
\end{equation}
and $x_n,y_n$ are the i.i.d. Gaussians described above. It follows that $\varphi \in H^{-s}(\T_R)$ almost surely for all $s>0$, where $H^{-s}(\T_R)$ denotes the Sobolev space of order $-s$ on $\T_R$.
By definition, the random variable $\varphi$ is the Gaussian Free Field on the circle, with covariance kernel (in the sense of distributions)
\[ \E[\varphi(\theta)\varphi(\theta')]=\log \frac{1}{|e^{i \theta/R}-e^{i \theta'/R}|}.\]
The probability measure $\P_{\T}$ induces a probability measure, still denoted $\P_\T$, on $H^{-s}(\T_R)$ by pushforward via the random variable $\varphi$. 
We denote by $H_0^{-s}(\T_R)$ the subspace of $f \in H^{-s}(\T_R)$ such that $\cjg f,1\cjd=0$. 
Note that, $\P_\T$-almost surely, $\varphi \in H_0^{-s}(\T_R)$.

The Gaussian Free Field $X_{\mc{C}^+_R}\in H^{-s}(\mc{C}_R^+)$ (for any $s>0$) on the half-cylinder $\mc{C}_R^+=\R_+ \times \T_R$ is defined as the sum of two independent Gaussian random variables taking values in $H^{-s}(\mc{C}_R^+)$: 
\[ X_{\mc{C}_R^+}=X_{\mc{C}_R^+,D}+P^R\varphi.\]
Here, $X_{\mc{C}_R^+,D}$ is the Dirichlet GFF, the Gaussian random variable with covariance kernel given by the Dirichlet Green function on $\mc{C}_R^+$
\[ \E[X_{\mc{C}_R^+,D}(t,\theta)X_{\mc{C}_R^+,D}(t',\theta')]= \log\frac{|1-e^{-\frac{t+t'-i(\theta-\theta')}{R}}|}{|e^{-\frac{t-i\theta}{R}}-e^{-\frac{t'-i\theta}{R}}|}, \]
and $P^R\varphi$ is the harmonic function on $\mc{C}_R^+$ with boundary value at $t=0$ given by the random variable $\varphi$ of \eqref{RVvarphi}, which vanishes as $t\to \infty$. It is given explicitly for $t\geq 0$, $\theta\in \T_R$ by 
\begin{equation}\label{PRvarphi}
P^R\varphi(t,\theta)=P\varphi(\frac{t}{R},\frac{\theta}{R}):=\sum_{n=1}^\infty 
\frac{x_n+iy_n}{2\sqrt{n}}e^{-n\frac{t}{R}}e^{in\frac{\theta}{R}}+\frac{x_n-iy_n}{2\sqrt{n}}e^{-n\frac{t}{R}}e^{-in\frac{\theta}{R}}.\end{equation}

We consider the state space $C(\R_+, H^{-s}(\mathbb{T}_R))\simeq C(\R_+,\R\times H^{-s}_0(\T_R))$ of continuous 
$\R\times H^{-s}_0(\T_R)$-valued functions on $\R_+$, equipped with its Borel $\sigma$-algebra. We moreover have the inclusion 
$C([0,T],H^{-s}(\T_R))\subset H^{-s}(\mc{C}_{R,T})$. 
The Markov process $\phi^R \in C(\mathbb{R}_+, H^{-s}(\mathbb{T}_R))$ will be given by the restrictions on time slices 
of the GFF on the cylinder $\mc{C}_R^+$:
\begin{equation}\label{defPhi^Rt}  
t\in \R_+\mapsto \phi^R(t):=c+X_{\mc{C}_R^+}(t,\cdot)\in H^{-s}(\T_R).
\end{equation}
Let us now describe this process more concretely using the Fourier decomposition in the $\T_R$ variable. 
First, any $L^2$ function on $\mc{C}_R^+$ can be decomposed under the form
\begin{equation}
\phi^R(t,\theta)
:=
\phi^R_0(t) + \varphi^R(t,\theta),
\end{equation}
where $\phi_0^R(t)\in \R$, $\varphi^R(t,\cdot)\in L^2(\T_R)$ with
\begin{align}
 \varphi^R(t,\theta) = \sum_{n=1}^\infty \frac{x_n^R(t)+iy_n^R(t)}{2\sqrt{n}}
e^{in \frac{\theta}{R}}+\frac{x_n^R(t)-iy_n^R(t)}{2\sqrt{n}}
e^{-in \frac{\theta}{R}}
\end{align}
for some $x_n^R(t),y_n^R(t)\in \R$. The same decomposition also applies to the random distribution $\phi^R=X_{\mc{C}_R^+}$ and a direct calculation gives that $\phi_0^R(t)$ is a standard Brownian motion $\phi_0^R(t)=B_{t/R}$ vanishing at $t=0$, with covariance $\E[B_tB_{t'}]=\min(t,t')$, while $\varphi^R(t,\theta)$ is a Gaussian random variable in $H^{-s}(\mc{C}_R^+)$ with covariance kernel 
\begin{equation}\label{covvarphiR} 
\E[ \varphi^R(t,\theta) \varphi^R(t',\theta')]= \log \frac{\max(e^{-t'/R}, e^{-t/R})}{|e^{-t'/R} e^{i \theta'/R}- e^{-t/R} e^{i \theta/R} |}.\end{equation}

Next, we describe these two processes in terms of their Fourier coefficients. 
Consider a probability space $\Omega$ with probability measure $\mathbb{P}$ that contains the random variables $(x_n, y_n)_{n \in \mathbb{N}^\ast}$ coming from the GFF on the circle (see \eqref{Gaussian_measure}), an independent Brownian motion $(B_t)_{t \geq 0}$ vanishing at $t=0$, and independent centred Gaussian processes$\{ (\tilde{x}_n(t))_{t \geq 0},(\tilde{y}_n(t))_{t \geq 0} \,|\,  n \in \mathbb{N}\}$ such that $\tilde{x}_n, \tilde{y}_m$ are independent for all $n,m$, and 
\[
\E[\tilde{x}_n (t) \tilde{x}_n (s) ]= \E[\tilde{y}_n (t) \tilde{y}_n (s) ]=  e^{- n  |t-s|}-e^{-n (t+s)}, \qquad \forall t,s \geq 0. 
\]
For $R>0$, let $t \mapsto Y_t^R(\theta)$ be the centred Gaussian process defined by the random Fourier series
\begin{equation}\label{DefY_t}
Y_t^R(\theta):=\sum_{n=1}^\infty\frac{\tilde{x}^R_n(t)+i\tilde{y}^R_n(t)}{2\sqrt{n}}e^{in\frac{\theta}{R}}+\frac{\tilde{x}^R_n(t)-i\tilde{y}^R_n(t)}{2\sqrt{n}}e^{-in\frac{\theta}{R}}
\end{equation} 
where $\tilde{x}_n^R(t):=\tilde{x}_n(t/R)$ and $\tilde{y}_n^R(t):=\tilde{y}_n(t/R)$.
A direct computation gives that $Y_t^R\in H^{-s}_0(\T_R)$ has covariance  
\begin{equation}\label{cov_YR}
\E [ Y_t^R(\theta) Y_{t'}^R(\theta') ] = \log \frac{|1- e^{-(t+t')/R} e^{i (\theta-\theta')/R}|}{|e^{-t'/R} e^{i \theta'/R}- e^{-t/R} e^{i \theta/R} |}- \frac{\min(t,t')}R.
\end{equation} 
We also notice that, in law, the following equality holds
\[X_{\mc{C}_R^+,D}(t,\theta)=B_{t/R}+Y_t^R(\theta).\] 
\begin{definition} \label{def: Markov process}
Let $R>0$. For $\varphi \in H^{-s}_0(\mathbb{T}_R)$ the random variable \eqref{RVvarphi}, let $\mathbb{P}^R_\varphi$ denote the law of the process\footnote{We do not distinguish the law of $(B_\bullet^R, \varphi_\bullet^R)$ and $B_\bullet^R+\varphi_\bullet^R$.} $B^R_\bullet + \varphi^R_\bullet \simeq (B^R_\bullet, \varphi^R_\bullet)$ conditioned to start at $(0,\varphi)$, where for all $t \geq 0$, $B^R_t := B_{t/R}$ and
\begin{align}
 & \varphi^R_t(\theta)=\sum_{n=1}^\infty \frac{x_n^R(t)+iy_n^R(t)}{2\sqrt{n}}
e^{in \frac{\theta}{R}}+\frac{x_n^R(t)-iy_n^R(t)}{2\sqrt{n}}
e^{-in \frac{\theta}{R}},\\
& x^R_n(t):=x_n e^{-n \frac{t}{R}}+ \tilde{x}^R_n(t), \qquad y^R_n(t):=y_n e^{-n \frac{t}{R}}+ \tilde{y}^R_n(t).
\end{align}
By \eqref{PRvarphi},   
\begin{align}\label{varphit_def}
\varphi^R_t(\theta)
=
P^R\varphi(t,\theta) + Y_t^R(\theta), \qquad \forall t \geq 0.
\end{align}
When we wish to describe the process started from $(c,\varphi) \in \mathbb{R}\times H^{-s}_0(\mathbb{T}_R)$, we will write $c+B^R_\bullet + \varphi^R_\bullet$. 
\end{definition}
The covariance of $P^R\varphi$ is given by 
\begin{align*}
& \E[P^R(\varphi)(t,\theta)P^R(\varphi)(t',\theta')]= -\log |1- e^{-(t+t')/R} e^{i (\theta-\theta')/R}|.
\end{align*}
Combining this with \eqref{cov_YR}, we obtain the following equality in law 
\[ X_{\mc{C}_R^+}(t,\theta)=B^R_{t}+ \varphi^R_t(\theta),\]
so that Definition \ref{DefY_t} matches with \eqref{defPhi^Rt} if we set $\phi^R(t)=c+B^R_{t}+ \varphi^R_t(\theta)$.
We also observe that for $(c,\varphi) \in \mathbb{R}\times H^{-s}_0(\mathbb{T}_R)$, the following scaling relation holds in law:
\begin{equation}\label{scalinrel} 
\phi^R(t, \theta)= \phi^1(\tfrac{t}{R}, \tfrac{\theta}{R}).
\end{equation}

We remark that $x_n^R(t),y_n^R(t)$ are independent Ornstein-Uhlenbeck processes with generator
\begin{equation*}
\mathcal{L}_n^R f(x^R_n)
=
\frac{n}{R} x^R_n f'(x^R_n)-\frac{n}{R} f''(x^R_n)= \frac{1}{R} ( n x^R_n f'(x^R_n)- n f''(x^R_n) ).
\end{equation*} 
This is discussed in further detail, along with the heuristics for the choice of the Markov process $\phi^R(t)$, in Appendix \ref{appendix: path integral}.

The generator of the Markov process $c+B_\bullet + \varphi_\bullet$ is given by the Free Hamiltonian  $\H^0$ that we will study in  Section \ref{sec:Hamiltonian}.

\subsection{Gaussian multiplicative chaos on $\mc{C}_R$}

We begin by defining approximations of the Gaussian multiplicative chaos (GMC) measures.
Let $R>0$ and $\gamma \in (0,2)$. We define the circle average
\begin{equation}\label{circle_average}
\varphi^{R,\eps}(t, \theta):= \frac{1}{2 \pi} \int_0^{2\pi} \varphi^{R}(t+\eps \cos(v),\theta+\eps \sin(v))  dv.
\end{equation}
Define $M_{R,\gamma,\eps}^{\circ,+}$ and $M_{R,\gamma,\eps}^{\circ, -}$ to be the random measures on $\mc{C}_R$ with density given by 
\begin{align*}
M^{\circ,+}_{R,\gamma,\eps}(dtd\theta)
&:=
 \eps^{\frac{\gamma^2}{2}}e^{\gamma (B_t^R+\varphi^{R,\eps} (t,\theta))} dtd\theta,
\\
M^{\circ, -}_{R,\gamma,\eps}(dtd\theta)
&:= \eps^{\frac{\gamma^2}{2}}e^{-\gamma (B_t^R+\varphi^{R,\eps}(t,\theta))} dtd\theta.
\end{align*}
Note that, since these measures are derived from the stochastic process $B_\bullet^R + \varphi_\bullet^R$, they depend on the choice of the initial condition  $\varphi \in H^{-s}_0(\mathbb{T}_R)$. 

The (regularised) GMC measures defined above satisfy an exact scaling relation. Indeed, the scaling relation \eqref{scalinrel} implies 
$\varphi^{R,\eps}(t, \theta)=\varphi^{1,\frac{\eps}{R}}(\frac{t}{R}, \frac{\theta}{R})$. Hence, we get for any measurable set $A$ that 
\begin{equation}\label{scalingGMC}
M^{\circ,+}_{R,\gamma,\eps}(A)= R^{\gamma Q} M^{\circ,+}_{1,\gamma, \frac{\eps}{R}}(R^{-1}A), 
\quad
 M^{\circ,-}_{R,\gamma,\eps}(A)= R^{\gamma Q} M^{\circ,-}_{1,\gamma, \frac{\eps}{R}}(R^{-1}A).
\end{equation}
All equalities hold in law. We will use this relation to deduce the general $R$ case from the $R=1$ case. For $R=1$, we shall 
omit the $R$ in the notations for the GMC measures following our general convention.

Before stating the proposition concerning the construction of the GMC measures, we introduce a Fourier regularised version of the GMC that is useful for the spectral theory. The limit will be the same as the limit constructed from the circle average cutoff. If 
$\varphi^R_t(\theta)=\sum_{n\in \Z}\varphi^R_n(t)e^{in\frac{\theta}{R}}$ we 
 let $\hat{\varphi}^{R,N}_t:=\sum_{|n|\leq N}\varphi^R_n(t)e^{in\frac{\theta}{R}}$ for $N\in \N$. 
 Define $M^{\wedge,+}_{R,\gamma,N}$ and $M^{\wedge,-}_{R,\gamma,N}$ to be the random measures on $\mc{C}_R$ with densities given by 
 \begin{align*}
M^{\wedge,+}_{R,\gamma,N}(dtd\theta)
&=
e^{\gamma (B_t^R + \hat{\varphi}_t^{R,N})-\frac{\gamma^2}{2}\E[ (\hat{\varphi}_t^{R,N})^2]} dtd\theta, 
\\
M^{\wedge,-}_{R,\gamma,N}(dtd\theta)
&=
  e^{-\gamma(B_t^R+\hat{\varphi}_t^{R,N})-\frac{\gamma^2}{2}\E[ (\hat{\varphi}_t^{R,N})^2]} dtd\theta.
\end{align*}
We observe that, as $\eps\to 0$ and $N\to \infty$,
\[ \E[ (\hat{\varphi}_t^{1,N})^2]=\log(N)+c_0+o(1), \quad  \E[ (\varphi_t^{1,\eps})^2]=\log(\eps^{-1})+o(1)\]
where $c_0$ is the Euler constant.
The existence of nontrivial limiting measures $\lim_{\eps\to 0}M^{\circ,\pm}_{R,\gamma,\eps}$ and $\lim_{N\to \infty}M^{\wedge,\pm}_{R,\gamma,N}$ is standard (see \cite{Kahane85} and the more recent reviews \cite{rhodes2014_gmcReview,Berestycki_lqggff}, as well as references therein) and using the scaling relation \eqref{scalingGMC}, one gets:
\begin{proposition} \label{prop: GMC construction}
Let $\gamma \in (0,2)$, $R>0$ and $ (c,\varphi) \in  \R \times H^{-s}_0(\T_R)$ for $s>0$. Then there exist random measures $M^+_{R,\gamma}$ and $M^-_{R,\gamma}$ such that the following convergence statements hold in probability under $\P_{\varphi} $ in the topology of vague convergence on the space of Radon measures on $(0,\infty) \times [0,2\pi]$.
\begin{itemize}
\item The  measures $M^{\wedge,\pm}_{R,\gamma,N}$ and $M^{\circ,\pm}_{R,\gamma,\eps}$ converge to $M^\pm_{R,\gamma}$ as $\eps\to 0$ and $N\to \infty$.
\item One has the scaling relation for each Borel set $A\subset \mc{C}_R$
\[ M^{+}_{R,\gamma}(A)= R^{\gamma Q} M^{+}_{1,\gamma}(R^{-1}A), 
\quad
 M^{-}_{R,\gamma}(A)= R^{\gamma Q} M^{-}_{1,\gamma}(R^{-1}A).\]
\end{itemize} 
\end{proposition}
The measures $M_{R,\gamma}^+$ and $M_{R,\gamma}^-$ constructed above are called GMC measures, and when $R=1$ we will simply write $M_\gamma^\pm=M^\pm_{1,\gamma}$. Under $\P_{\varphi} $ the total mass of the measure $M_{\gamma}^\pm((0,T)\times [0,2\pi])$, or any bounded set that is arbitrarily close to the boundary $\{0\}\times[0,2\pi]$, can be infinite for any $T>0$. As we shall see below from the moment bounds of Lemma \ref{lem: GMC positive moments}, we obtain that, almost surely in the variable $\varphi$, the mass is finite (i.e. under the unconditional probability measure $\P$ with expectation $\E$).

We record below several useful moment bounds for GMC that will be used in the sequel. The first lemma concerns positive moments and its proof can be found in \cite[Prop. 3.3]{Robert_Vargas}.
\begin{lemma}\label{lem: GMC positive moments}
Let $\gamma \in (0,2)$. For any $A \subset [0,\infty)\times [0,2\pi]$ compact Borel set,
\begin{equation*}
\mathbb{E}[M^\pm_{\gamma}(A)^p] < \infty,
\qquad
\forall p \in (0, 4/\gamma^2).  
\end{equation*}
\end{lemma}
The next lemma concerns negative moments; see \cite[Prop. 3.6]{Robert_Vargas}.
\begin{lemma}\label{lem: GMC negative moments}
Let $\gamma \in (0,2)$. For any $A \subset (0,\infty)\times [0,2\pi]$ bounded (nonempty) open set,
\begin{align*}
\E[M^\pm_{\gamma}(A)^{-p}]
<
\infty, \qquad \forall p \in (0,\infty).    
\end{align*}
\end{lemma}

\subsection{Connection with the GFF and GMC on $\mathbb{C}$}

The GFF is invariant in law under conformal maps \cite[Theorem 1.57]{Berestycki_lqggff}. Therefore, using the conformal maps $\ell_R:\mathcal{C}_{R}^+ \rightarrow \mathbb{D}\setminus\{0\}$, where $\ell_R(t,\theta) = e^{-t/R+i\theta/R}$ and $\mathcal{C}_{R}^+ = \R_+\times\mathbb{T}_R$, we may relate the GFF and the GMC on $\mathcal{C}_{R}^+$ to the corresponding objects on the unit punctured disk $\mathbb{D}\setminus \{0\} = \{ z \in \mathbb C : 0 < |z| < 1\}$ or the Riemann sphere. This is convenient since it allows us to place our objects within the framework of existing works on Liouville conformal field theory \cite{DKRV16, GKRV21_Segal, GKRV} and will allow us to appeal to arguments and techniques in those papers in a more direct way. We will fix $R=1$ as we will mainly work in this setting.

Let us first explain the conformal invariance of the GFF. Let  $X_{\D}$ denote the Gaussian Free Field on the unit disk $\D=\{z\in \C\,|\, |z|\leq 1 \}$ with Dirichlet condition boundary conditions, i.e.\ with covariance 
\[ \E[ X_{\D}(x)X_{\D}(x')]=\log\frac{|1-x\bar{x}'|}{|x-x'|}.\]
Then the random variable $\phi^1(t,\theta)$ defined in \eqref{defPhi^Rt} with $R=1$ is related to $X_{\D}$ by the following equality in law which holds under the conditional probability measure $\mathbb{P}_\varphi$ for every $\varphi \in H^{-s}_0(\mathbb{T})$ and every $c \in \mathbb{R}$:
\begin{equation}\label{linkGFFdisk}
\phi^1(t,\theta)=c+X_\D(e^{-t+i\theta})+P_\D\varphi(e^{-t+i\theta})
\end{equation}
where $P_\D \varphi$ is the harmonic extension of $\varphi$ on the unit disk $\D$.
In particular, the equality in law holds also under the stationary measure $\mu_0$.

Finally, let us make the connection with GMC on the complex plane $\C$. 
The circle average \eqref{circle_average} with $R=1$ becomes, after addition of the Brownian $B_t$, at first order as $\eps \to 0$ a regularization on the plane GFF restricted to $\D$, $X=P_\D\varphi+X_\D$,  by 
\[ X_\eps=\frac{1}{2\pi} \int_{0}^{2\pi} X(e^{-t+i\theta}+\eps e^{-t+i\theta+iv})dv\]
and it is a regularization on a geodesic circle of radius $\eps$ for the metric $|dz|^2/|z|^2=(\ell_1)_*(dt^2+d\theta^2)$. 
As explained in \cite[Proposition 3.4]{GRV} or \cite[Theorem 2.8]{Berestycki_lqggff}, if $M_\gamma^{\pm,\C}$ denotes the classical GMC regularised using the GFF averaged over Euclidean circles  $X^\circ_\eps$ (i.e. using the background metric $|dz|^2$), i.e.\ the random measure obtained from the limit
\begin{align*}
M_{\gamma}^{\pm,\C}(dz)
=   
\lim_{\eps \rightarrow 0}\eps^{\frac{\gamma^2}2}e^{\pm \gamma X^\circ _\eps}dz
\end{align*}
 we have in $\D \setminus \{0\}$
\begin{equation}\label{linkGMC} 
(\ell_1)_*M^{\pm}_{\gamma}(dtd\theta)=|z|^{-\gamma Q}M^{\pm,\C}_\gamma(dz).
\end{equation}
Sometimes, if we want to make the dependence on the underlying field explicit in the GMC, we write $M_\gamma^{\sigma,\C}(X,dz)$.

\section{The Sinh-Gordon Hamiltonian on the infinite cylinder}\label{sec:Hamiltonian}

In this section, we prove Theorem \ref{thm:spectrum_H}. In particular, we introduce and study the Sinh-Gordon Hamiltonian on the cylinder $\mc{C}_1$, viewing it as the generator of a Markovian dynamics, in a way similar to the Liouville Hamiltonian studied in \cite{GKRV}. As stated before, the restriction $R=1$ (which we will shall assume throughout the section unless stated otherwise) will be removed by appealing to exact scaling relations derived from \eqref{scalinrel}. Unless stated otherwise, we fix the parameters $\mu>0$ and $\gamma\in(0,2)$ appearing in \eqref{path_integral_sinh} and drop them from the notation.

\subsection{Hilbert space and Free Hamiltonian}

We first recall the construction of the Hilbert space and definition and basic properties of the Free Hamiltonian. Recall that $\P_\T$
 is the probability measure on the sequence space $\Omega_\T := (\R^2)^{\N^*}\simeq H_0^{-s}(\mathbb{T})$ defined by
\begin{equation}
\P_\T
=
\bigotimes_{n \in \N^*} \frac{1}{2\pi} e^{-\frac 12(x_n^2 + y_n^2)} dx_n dy_n,
\end{equation}
where the identification is defined by the mapping of an i.i.d. Gaussian sequence $(x_n,y_n)$ to an element $\varphi \in H^{-s}_0(\mathbb{T})$ defined by its Fourier series
\begin{equation}\label{def_varphi}
\varphi(\theta)
=
\sum_{n \neq 0} \varphi_{n} e^{ i n \theta},  \qquad
\textrm{where } \varphi_{n} 
:= 
\frac{x_n + i y_n}{2 \sqrt n} , \qquad \varphi_{-n} = \overline{\varphi_{n}}.
\end{equation}

The Hilbert space of the theory is 
\[ \mc{H}:=L^2(\R\times \Omega_\T,dc\otimes d\P_\T)\simeq L^2(H^{-s}(\T),\mu_0)\]
 where $dc$ denotes the standard Lebesgue measure on $\mathbb{R}$, and
  $\mu_0$ is the push-forward of $dc\otimes d\P_\T$ by the random variable $c+\varphi$ with $\varphi\in H_0^{-s}(\T)$. Scalar products on this Hilbert space will be denoted $\langle \cdot, \cdot \rangle_{\mc{H}}$. 

We now rigorously define the Free Hamiltonian following \cite[Section 4]{GKRV}. Let $\mathcal{S} \subset L^2(\Omega_\T)$ denote the linear span of functions of the form $F = F(x_1,y_1, \dots, x_N,y_N) \in C^\infty((\R^2)^N)$ for some $N \in \N$ and such that all its partial derivatives have at most polynomial growth at infinity. Let $\bP$ denote the operator acting on $\mathcal{S}$ defined by
\begin{equation}
\bP
=
\sum_{n=1}^\infty n(\X_n^* \X_n + \Y_n^* \Y_n),	
\end{equation}
where, for $n \in \N$,  
\begin{equation}
\X_n = \partial_{x_n}, \qquad \X_n^* = - \partial_{x_n}+x_n, \qquad 
\Y_n = \partial_{y_n}, \qquad \Y_n^* = - \partial_{y_n} + y_n. 	
\end{equation}
Here the adjoints are the formal adjoints on $\mc{S}$ with respect to the Gaussian measure $\P_\T$. The operator $\bP$ admits a unique self-adjoint extension (unbounded) on $L^2(\Omega_\T)$. It has discrete spectrum and an orthonormal  basis of  eigenfunctions given by generalised Hermite polynomials, see \cite[Section 4.1]{GKRV}. 

In the following proposition, we recall the construction and main properties of ${\bf H}^0$ 
from  \cite[Propositions 4.3 \& 4.4 \& 4.5]{GKRV} (For proofs, see the reference.). To state it, we introduce the following quadratic form $\mc{Q}_0$ defined as
\[ \mc{Q}_0(u,v):=\frac{1}{2}\int_\R \E[ \pl_cu \cdot \pl_c\bar{v}+2({\bf P}u)\bar{v}]dc\]
on the dense subspace 
\begin{equation}\label{def_mcE}
\mc{E}:=\{ \psi(c)F\,|\, \psi\in C_c^\infty(\R), F\in \mc{S}\}
\end{equation} 
of $L^2(\R\times \Omega_\T)$. Its domain is 
\[\mc{D}(\mc{Q}_0)=\{ u\in \mc{H}\,|\, \mc{Q}_0(u,u)<\infty \}.\]

\begin{proposition}[\textbf{Free Hamiltonian}]\label{prop:FreeHamiltonian}
1) The quadratic form $\mc{Q}_0$ is closable with domain $\mc{D}(\mc{Q}_0)$ and generates an unbounded, self-adjoint, positive operator
\begin{equation}
\H^0:=-\frac 12 \frac{d^2}{dc^2} + \bP	
\end{equation}
on the domain $\mc{D}(\H^0)=\{u\in \mc{D}(\mc{Q}_0)\,|\, \exists C>0, \forall v\in \mc{D}(\mc{Q}_0), 
|\mc{Q}_0(u,v)|\leq C\|v\|_{\mc{H}}\}$.\\ 
2) The operator ${\bf H}^0$ generates a  strongly continuous contraction semigroup on $\mc{H}$ such that, for all $t > 0$:
\[
e^{-t\H^0}f(c,\varphi)= \E_{\varphi}[f(c+B_t+\varphi_t)],
\qquad \forall c \in \R, \forall \varphi \in H^{-s}_0(\T), \forall f \in \mc{H},	
\]
where $\E_\varphi$ denotes the expectation conditional on $\varphi$ and where $\varphi_\bullet$ is as in \eqref{varphit_def}.

3) The propagator $e^{-t{\bf H}^0}$ extends to a continuous semigroup on $L^p(\R\times \Omega_\T)$ for all $p\in [1,\infty]$ with norm $\leq 1$ and is strongly continuous when $p<\infty$.
\end{proposition}

Notice that, since $\mathbf{H}^0$ is self-adjoint, $\mu_0$ is an invariant measure of the (shifted) process $(\cdot + B_\bullet,\varphi_\bullet)$. 
We will also sometimes denote $\varphi (t, \theta)$ in place of $\varphi_t (\theta)$; in the first case we view $\varphi_t$ as living on a cylinder and in the second we view $\varphi_t$ as a Markov flow.

From \eqref{linkGFFdisk}, notice that  $e^{-t{\bf H}^0}f$ can also be written, for $\varphi\in H_0^{-s}(\T)$ with $s>0$ and $c\in \R$, as 
\[ e^{-t{\bf H}^0}f(c,\varphi)=\E_{\varphi}[f(c+X\circ e^{-t}|_{\T})]\]
where $X=P\varphi+X_\D$ is the Gaussian Free Field on $\D$, decomposed using the Dirichlet GFF $X_\D$ and the harmonic extension $P\varphi$ of $\varphi$ in $\D$.

\subsection{The Sinh-Gordon Hamiltonian} \label{sec:Sinh_hamiltonian}

To define the Sinh-Gordon Hamiltonian ${\bf H}$, we shall proceed as in \cite[Section 5]{GKRV} by viewing ${\bf H}$ as the generator of a Markov semigroup and showing that it also coincides with the Friedrichs extension associated to a quadratic form $\mc{Q}$. As for the Liouville theory, the potentials appearing in the Hamiltonian are defined by GMC theory. When $\gamma < \sqrt 2$, the potentials can be defined using GMC on the circle. However, when $\gamma>1$, they are not in $L^2(\Omega_\T)$ in the $\varphi$ variable, which complicates spectral-theoretic questions, such as the determination of the domain of ${\bf H}$. The case of $\gamma \in [\sqrt{2},2)$ is even more subtle as we cannot use the GMC on the circle anymore to define the potentials.

First, we define the operator ${\bf T}_t$: for $f$ bounded and continuous on $H^{-s}(\T)$ for $s>0$ and for $t\geq 0$, we set (recall that $\gamma Q=2+\frac{\gamma^2}{2}$)
\begin{equation}\label{propagatorT_t}
\begin{split}
{\bf T}_tf(c,\varphi):= \E_{\varphi}\Big[f(c+X\circ e^{-t}|_\T)e^{-\mu\int_{\A_t}|x|^{-\gamma Q}(e^{\gamma c}M_\gamma^{+,\C}(dx)+e^{-\gamma c}M^{-,\C}_\gamma(dx))} \Big],
\end{split}
\end{equation} 
where $\A_t:=\D\setminus e^{-t}\D$  and $X=P\varphi+X_\D$ is the Gaussian Free Field on $\D$ with boundary condition $\varphi$ at $\pl \D=\T$.
The definition is justified by the relation  \eqref{linkGMC} and will yield a Feynman-Kac representation of the semigroup in terms of a potential $V$ appearing in the Hamiltonian generating this semigroup.

The approach we use is via a regularization of the potential. 
For $k \geq 0$ and $\varphi \in H_0^{-s}(\T)$ for $s>0$ let
\begin{equation}
\varphi^{(k)}(\theta)
=
\sum_{|n| \leq k} \varphi_n e^{i n \theta}\in C^\infty(\T).
\end{equation}
We  define the potentials $V^{(k)}_+, V^{(k)}_-: H^{-s}_0(\T) \rightarrow \R_+$ (for any $s>0$ fixed)
by 
\begin{equation}
\begin{split}
V_{\pm}^{(k)}(\varphi)=& 
\int_0^{2\pi} e^{\pm \gamma \varphi^{(k)}(\theta) - \frac{\gamma^2}2 \E[\varphi^{(k)}(\theta)^2]} d\theta,  
\end{split}
\end{equation}
Viewed as multiplication operators on $\mc{H}$, $e^{\pm \gamma c}V^{(k)}_\pm$ are unbounded, positive, 
and symmetric on the dense subspace $\mc{E}$. Moreover, by \cite{Kahane85,rhodes2014_gmcReview}, for all $\gamma\in (0,\sqrt{2})$ 
the following limit exists $\P_\T$ almost surely 
\begin{equation}\label{Upm} 
V_\pm:=\lim_{k\to \infty}V_{\pm}^{(k)} \in L^{p}(\Omega_\T), \forall p<2/\gamma^2.
\end{equation}

\begin{lemma}\label{semigroupT_t}
The family of operators ${\bf T}_t$ for $t\geq 0$ extends to a self-adjoint contraction semigroup on $\mc{H}$ with norm $\|{\bf T}_t\|_{\mc{L}(\mc{H})}\leq 1$. 
When $\gamma<\sqrt{2}$, the classical Feynman-Kac representation for Schr\"odinger operators holds true: 
\begin{equation}\label{TtFeynmannKac}
 {\bf T}_tF(c+\varphi)=\E[ F(c+B_t+\varphi_t)e^{-\mu \int_{0}^t (e^{\gamma(c+B_s)}V_+(\varphi_s)+e^{-\gamma(c+B_s)}V_-(\varphi_s))ds}] \end{equation}
where $V_\pm$ are the potentials defined in \eqref{Upm}. Finally, ${\bf T}_t$ extends as a strongly continuous semigroup on $L^p(H^{-s}(\T),\mu_0)$ for all $p\in [1,\infty]$.
\end{lemma}
\begin{proof}
The norm estimate follows from the almost sure bound
\[ \exp\Big(-\mu  \int_{\A_t}|x|^{-\gamma Q}(e^{\gamma c}M_\gamma^{+,\C}(dx)+e^{-\gamma c}M^{-,\C}_\gamma(dx))\Big)\leq 1\]
and the fact that $\|e^{-t{\bf H}^0}\|_{\mc{L}(\mc{H})}\leq 1$. The self-adjoint property is direct, using that $e^{-t{\bf H}^0}$ is self-adjoint and that the potential is real valued.
To prove it is a semigroup, we first use the Markov property and the conformal invariance of the GFF on $\D$, 
\begin{equation}\label{GFFscaling} 
X_\D(e^{-t}x)=\tilde{X}_{\D}(x)+P(X_\D \circ e^{-t}|_{\T})(x),
\end{equation}
where $\tilde{X}_\D$ is an independent Dirichlet GFF on $\D$. We have 
\[\begin{split}
{\bf T}_t({\bf T}_sF)(c+\varphi)=& \E_\varphi[({\bf T}_sF)(c+X\circ e^{-t}|_{\T})e^{-\mu\sum_{\sigma = \pm 1}\int_{\A_t} |x|^{-\gamma Q}e^{\sigma \gamma c}M_\gamma^{\sigma,\C}(X,dx)}]\\
=&  \E_\varphi\Big[\E_{X_{\D}\circ e^{-t}|_{\T}}\Big[ F\Big((c+\tilde{X}_{\D}\circ e^{-s}+P(X_{\D}\circ e^{-t}|_{\T})\circ e^{-s}+P\varphi\circ e^{-t-s})|_{\T}\Big)\\
& \qquad \qquad\qquad  \times e^{-\mu\sum_{\sigma = \pm 1}\int_{\A_s} |x|^{-\gamma Q}e^{\sigma \gamma c}M_\gamma^{\sigma,\C}(\tilde{X}_\D+P((X_\D+P\varphi)\circ e^{-t}|_{\T}),dx)}\Big]\\
& \qquad \qquad\qquad\times e^{-\mu\sum_{\sigma = \pm 1}\int_{\A_t} |x|^{-\gamma Q}e^{\sigma \gamma c}M_\gamma^{\sigma,\C}(X,dx)}\Big],
\end{split}\]
where we use the notation $M_\gamma^{\pm}(X,dx)$ to emphasize that the GMC is defined from the GFF $X$. 
Using \eqref{GFFscaling} and $P((P\varphi)\circ e^{-t}|_{\T})=P\varphi$, we obtain, using the conformal covariance of the GMC \cite[Theorem 2.8]{Berestycki_lqggff}, that a change of variables $x=e^{t}y$ in the $\A_{s}$ integral produces
\[ \int_{\A_s}|x|^{-\gamma Q}e^{\pm \gamma c}M_\gamma^{\pm,\C}(\tilde{X}_\D+P((X_\D+P\varphi)\circ e^{-t}|_{\T}),dx)= \int_{e^{-t}\A_s}|y|^{-\gamma Q}e^{\pm \gamma c}M_\gamma^{\pm,\C}(X,dy)\]
and 
\[ \int_{\A_{t+s}}|x|^{-\gamma Q}e^{\pm \gamma c}M_\gamma^{\pm,\C}(dx)= 
\int_{\A_t}|x|^{-\gamma Q}e^{\pm \gamma c}M_\gamma^{\pm,\C}(dx)+ \int_{e^{-t}\A_{e^{-s}}}|x|^{-\gamma Q}e^{\pm \gamma c}M_\gamma^{\pm,\C}(dx).\]
Using again \eqref{GFFscaling}, we get 
\[\begin{split}
{\bf T}_t({\bf T}_sF)(c+\varphi)=&
 \E_\varphi\Big[ F\Big((c+X_{\D}\circ e^{-s-t}+P\varphi\circ e^{-t-s})|_{\T}\Big)e^{-\mu\sum_{\sigma = \pm 1}\int_{\A_{s+t}} |x|^{-\gamma Q}e^{\sigma \gamma c}M_\gamma^{\sigma,\C}(X,dx)}\Big]\\
 =& ({\bf T}_{t+s}F)(c+\varphi).
\end{split}\]
This completes the proof of the semigroup property. 

The Feynman-Kac formula is proved by rewriting $X|_{e^{-t}}=B_t+\varphi_t$ 
and decomposing the integral on $\A_t$ using the radial coordinates $x=e^{-s+i\theta}$ as in \eqref{linkGMC}. See the proof of 
\cite[Proposition 5.1]{GKRV} for more details.

For the extension on $L^p$, we simply use $\|{\bf T}_tF\|_{L^p}\leq \|e^{-t{\bf H}^0}(|F|)\|_{L^p}$ thus the result follows from the property of $e^{-t{\bf H}^0}$ in Proposition \ref{prop:FreeHamiltonian}.
\end{proof}

Being a self-adjoint contraction semigroup on $\mc{H}$, we deduce from the Hille-Yosida theorem that ${\bf T}_t=e^{-t{\bf H}_*}$ for some 
 generator ${\bf H}_*$, a positive self-adjoint operator with domain $\caD({\bf H}_*)$ consisting of $\psi\in \mc{H}$ 
 such that $\lim_{t\to 0}\frac{1}{t}(e^{-t{\bf H}_*}-1)\psi$ exists in $\mc{H}$. There is a quadratic form $\mc{Q}_*$ associated to ${\bf H}_*$ on the domain $\{ u\in \mc{H}\,|\, \lim_{t\to 0^+}t^{-1}\cjg u, (u-e^{-t{\bf H}_*}u)\cjd_{\mc{H}}<\infty \}$, and $\mc{Q}_*$ is closed (see  \cite[Chapter 1, Lemma 4.2]{S98}).
We are going to show that ${\bf H}_*$ is a Friedrichs extension associated to an explicit quadratic form $\mc{Q}$ that we shall define below using the GMC measure. First, we define:
\begin{definition}
For $k > 0$ let $\mc{Q}^{(k)}$ denote the symmetric bilinear form
\begin{equation}
\mc{Q}^{(k)}(F,G)
=
\int_\R \E \left[ \frac 12\partial_c F\partial_c \bbar{G} + (\bP F)\bbar{G}+ \mu  (e^{\gamma c}V_+^{(k)}+e^{-\gamma c}V_-^{(k)}) F \bbar{G} \,\right] dc , 
\qquad \forall F,G \in \mathcal{E},
\end{equation}
where $\mc{E}$ is defined in \eqref{def_mcE}, and let $\mc{D}(\mc{Q}^{(k)})$ be its domain, defined as the completion of $\mc{E}$ for $\mc{Q}^{(k)}$.
\end{definition}
The domain $\mc{D}(\mc{Q}^{(k)})$ injects in $\mc{H}$ and $\mc{Q}^{(k)}$ being positive, the Friedrichs extension provides a self-adjoint operator ${\bf H}^{(k)}$ on the domain 
\[ \mc{D}({\bf H}^{(k)})=\{ F\in \mc{D}(\mc{Q}^{(k)})\,|\, \exists C>0, \forall G \in \mc{D}(\mc{Q}^{(k)}),\,  
|\mc{Q}^{(k)}(F,G)|\leq C\|G\|_{\mc{H}}\}\]
and ${\bf H}^{(k)}$ is defined on this domain by 
\[ \cjg {\bf H}^{(k)}F,G\cjd_{\mc{H}}=\mc{Q}^{(k)}(F,G), \quad \forall G \in \mc{D}(\mc{Q}^{(k)}).\]
On the dense subspace $\mc{E}$ it is given by 
\[  {\bf H}^{(k)}={\bf H}^0+\mu e^{\gamma c}V_+^{(k)}+\mu e^{-\gamma c}V_-^{(k)},\]
where (for $\varphi_s$ being the random process defined by \eqref{varphit_def})
\[ V_\pm^{(k)}(\varphi_s)=\int_0^{2\pi}e^{\gamma \varphi_s^{(k)}(\theta)-\frac{\gamma^2}{2}\E[(\varphi_s^{(k)}(\theta))^2]}d\theta.\]

Now, we follow the strategy of \cite[Section 5.2]{GKRV}, and since the proofs are very similar, we do not repeat the details.
First, by the same argument as \cite[Proposition 5.3]{GKRV}, the semigroup $e^{-t{\bf H}^{(k)}}$ satisfies the Feynman-Kac formula:
\begin{equation}\label{FK_Hk}
 e^{-t{\bf H}^{(k)}}F(c,\varphi)=\E_{\varphi}[ F(c+B_t+\varphi_t)e^{-\mu\int_0^t (e^{\gamma (c+B_s)}V_+^{(k)}(\varphi_s)+e^{-\gamma (c+B_s)}V_-^{(k)}(\varphi_s))ds}].
 \end{equation}
For $F,G\in \mc{E}$, we define the quadratic form 
\begin{equation}\label{Q_as_limit}
\mc{Q}(F,G):=\lim_{k\to \infty}\mc{Q}^{(k)}(F,G).
\end{equation}
To see that the limit exists, we rewrite $\mc{Q}^{(k)}(F,G)$ using the Cameron-Martin theorem: write 
$F(c,\varphi)=F(c,x_1,y_1,\dots,x_n,y_n)$ and $G(c,\varphi)=G(c,x_1,y_1,\dots,x_n,y_n)$ and remark that for $k\geq n$, we have 
\begin{equation}\label{shift}
 \int_{\R} \E[ e^{\gamma c}V_+^{(k)}(\varphi)F\bbar{G}+e^{-\gamma c}V_-^{(k)}(\varphi)F\bbar{G}]dc =\int_{\R}\int_0^{2\pi} 
\E[ e^{\gamma c}F_{+}\bbar{G}_++e^{-\gamma c}F_{-}\bbar{G}_- ]dc d\theta, 
\end{equation}
where 
\[ F_{\pm}(\theta,c,x_1,y_1,\dots,x_n,y_n)=F(c,x_1\pm \gamma\cos(\theta),y_1\mp \gamma\sin(\theta),\dots,x_n\pm \frac{\gamma}{\sqrt{n}}\cos(n\theta),y_n\mp \frac{\gamma}{\sqrt{n}}\sin(n\theta)).
\]
In particular, we see that \eqref{shift} is independent of $k$ as long as $k\geq n$, and the limit\footnote{This holds for all $\gamma > 0$, but the closability of the form, i.e.\ Lemma \ref{lem:Q_equal_Qstar} uses that $\gamma \in (0,2)$.} \eqref{Q_as_limit} exists and is finite. 
The following then holds.

\begin{lemma}\label{lem:Q_equal_Qstar}
The quadratic forms $\mc{Q}$ and $\mc{Q}_*$ coincide on $\mc{E}$; in particular $\mc{Q}$ is closable.
\end{lemma}
\begin{proof} The proof is essentially the same as that of \cite[Lemma 5.4]{GKRV}; we briefly recall the main steps for the reader's convenience  and refer there for more details. Recall that $Q:=\gamma/2+2/\gamma$. One has to prove that for $F,G\in \mc{E}$, 
\[\begin{split} 
D_t:=&\int_{\R} \E[ F((c+X_{\mathbb D}\circ e^{-t} +P\varphi \circ e^{-t})|_{\mathbb T})G(c+\varphi)e^{-\mu\sum_{\sigma = \pm 1}\int_{\A_t}|x|^{-\gamma Q}e^{\sigma \gamma c}M_\gamma^{\sigma,\C}(dx)}]dc\\
=&\cjg F,G\cjd_{\mc{H}}-t\mc{Q}(F,G)+o(t).
\end{split}\]
The idea is to use the free field propagator: let $W_t:=\sum_{\sigma = \pm 1}e^{\sigma \gamma c}\int_{\A_t}|x|^{-\gamma Q}M_\gamma^{\sigma,\C}(dx)=:W_t^++W_t^-$, then
\[ \begin{split}
D_t=&\cjg (e^{-t{\bf H}^0}-{\rm Id})F,G\cjd_{\mc{H}}+ (1+o(1))\int_{\R} \E[ F((c+X_{\mathbb D}\circ e^{-t} +P\varphi \circ e^{-t})|_{\mathbb T})G(c+\varphi)(e^{-\mu W_t}-1)]dc \\
=& -t\mc{Q}_0(F,G)
\\
&+(1+o(1))\int_{\R} \E[ (F((c+X_{\mathbb D}\circ e^{-t} +P\varphi \circ e^{-t})|_{\mathbb T})-F(c+\varphi))G(c+\varphi)(e^{-\mu W_t}-1)]dc\\
& - (1+o(1))\int_{\R} \E[ F(c+\varphi)G(c+\varphi)(1-e^{-\mu W_t})]dc\\
=:& -t\mc{Q}_0(F,G)+R_t-S_t.
\end{split}\]

We claim that the term $S_t$ has the following asymptotic behaviour as $t\to 0$
\[ S_t=t(1+o(1))\mu \int_{\R}( e^{\gamma c}\E[F_+G_+]+e^{-\gamma c}\E[F_-G_-])dc.\] 
We follow closely the proof of \cite[Lemma 5.4]{GKRV}. It suffices to assume $F,G$ positive. 
First the upper bound is simply a consequence of the bound $1-e^{-x}\leq x$ to reduce to an estimate on $\int_{\R} \E[ F(c+\varphi)G(c+\varphi)W_t]dc$
and then we apply the  estimate of \cite[Lemma 5.4]{GKRV} for each term $W_\pm^t$ to get 
\[ S_t\leq t(1+o(1))\mu \int_{\R}( e^{\gamma c}\E[F_+G_+]+e^{-\gamma c}\E[F_-G_-])dc.\] 
The lower bound uses the inequality $(1-e^{-x})\geq xe^{-x}$ and reduces the analysis to estimating 
\begin{align*}
\int_{\R} \E[ FG W_t]dc+\int_{\R} \E[ FGW_t(e^{-\mu W_t}-1)]dc.
\end{align*}
As above, the first term is equal to $t(1+o(1))\mu \int_{\R}( e^{\gamma c}\E[F_+G_+]+e^{-\gamma c}\E[F_-G_-])dc$.
To prove that the second term is $o(t)$, the argument in the proof of \cite[Lemma 5.4]{GKRV} tells us that it suffices to show that there is $q>1$
such that the function 
\[f^\pm_t(c,r):=\E[ |1-e^{\mu Z^\pm_t(re^{i\theta})}|^q]^{1/q}\]
goes to $0$ uniformly as $t\to 0$, where 
\[Z^\pm_t(re^{i\theta}):=\sum_{\sigma=\pm 1}
e^{\sigma\gamma c}\int_{\A_t}|x|^{-\gamma Q}|x-re^{i\theta}|^{\mp \sigma \frac{\gamma^2}{2}}M^{\sigma,\C}_\gamma(dx).\] 
This function is increasing as function of $t$.
We can regularize it at scale $\delta>0$ small, by setting $f^{\pm,\delta}_t(c,r):=\E[ |1-e^{\mu Z^{\pm,\delta}_t(re^{i\theta})}|^q]^{1/q}$
with $Z^{\pm,\delta}_t(re^{i\theta})$ defined as $Z^{\pm}_t(re^{i\theta})$ but replacing the powers $|x-re^{i\theta}|^{\mp\sigma \frac{\gamma^2}{2}}$ by
 $\max(|x-re^{i\theta}|,\delta)^{\mp\sigma \frac{\gamma^2}{2}}$; the function $f^{\pm,\delta}_t$ is continuous. Then the proof is mutatis mutandis like in the 
 proof of \cite[Lemma 5.4]{GKRV}: it shows that $f^{\pm,\delta}_t\to f^{\pm}_t$ uniformly as $\delta\to 0$
  once we have observed that for $\alpha\in (0,1)$ such that $\alpha q<\frac{2}{\gamma}(Q-\gamma)$ (to ensure that the expectation is finite using \cite[Lemma 3.10]{DKRV16}) and $\alpha q\leq 1$
 \[\begin{split} 
 \| f^\pm_{t}(c,r)- f^{\pm,\delta}_{t}(c,r)\|\leq & \mu \E[ |Z^\pm_t(re^{i\theta})-Z^{\pm,\delta}_t(re^{i\theta})|^{\alpha q}]^{1/q}\\
\leq &  \sum_{\sigma=\pm 1}\mu e^{\sigma \alpha \gamma c}\E[ |Z^{\pm,\sigma}_t(re^{i\theta})-Z^{\pm,\sigma,\delta}_t(re^{i\theta})|^{\alpha q}]^{1/q},
 \end{split}\]
 where $Z^{\pm,\sigma}_t=\int_{\A_t}|x|^{-\gamma Q}|x-re^{i\theta}|^{\mp \sigma \frac{\gamma^2}{2}}M^{\sigma,\C}_\gamma(dx)$ and similarly for $Z^{\pm,\sigma,\delta}_t$. Indeed, each term $\sigma=\pm 1$ above has the same property to those treated in  \cite[Lemma 5.4]{GKRV} (here the $c$ variable is varying in a compact set due to our assumption on $F,G$, it is thus harmless). All this shows that $f_t^\pm$ are continuous and go pointwise to $0$ as $t\to 0$, but since $f_t^\pm$ is decreasing as $t\to 0$, Dini's theorem shows that the convergence is uniform with respect to the variables $c,r$.

Next for the term $R_t$, using that $1-e^{-x}\leq x$, we get 
\[ |R_t| \leq \mu \sum_{\sigma=\pm 1}\int_\R \E[ |F_\sigma(c+B_t+\varphi_t)-F_\sigma(c+\varphi)|\times |G_\sigma(c+\varphi)|W_t^{\sigma}]dc\]
and  the estimates of the proof of \cite[Lemma 5.4]{GKRV} can be applied directly to each term $\sigma = \pm 1$, giving $|R_t|=o(t)$. 
\end{proof}

Let $\mc{D}(\mc{Q})$ be the closure of $\mc{E}$ for the norm induced by $\mc{Q}$; note that $\mc{D}(\mc{Q})\subset \mc{H}$. The same argument as in \cite[Proposition 5.5.]{GKRV} shows 
\begin{lemma}\label{lem:Q_equal_Qstar}
For $\gamma\in (0,2)$, the quadratic form $\mc{Q}$ with domain $\mc{D}(\mc{Q})$ defines a self-adjoint operator ${\bf H}$ with domain denoted $\mc{D}({\bf H})\subset \mc{D}(\mc{Q})$ by $\mc{Q}(F,G)=\cjg {\bf H}F,G\cjd_{\mc{H}}$ for $F\in \mc{D}({\bf H})$ and $G\in \mc{D}(\mc{Q})$, and ${\bf H}={\bf H}_*$.
\end{lemma}
Here the only difference in the proof, compared to  \cite[Proposition 5.5.]{GKRV}, is that the Cameron-Martin shifts contain two terms rather than one single term, but each term is already dealt with in  the proof of 
\cite[Proposition 5.5.]{GKRV}, using the Feynman-Kac representation \eqref{FK_Hk}.

\subsection{Compact resolvent and discrete spectrum}
In this section, we will show that the resolvent of ${\bf H}$ is compact and thus ${\bf H}$ has purely discrete spectrum. 

We first describe mapping properties of the propagator ${\bf T}_t=e^{-t{\bf H}}$. This will be used to deduce estimates on the eigenfunctions of ${\bf H}$.
\begin{lemma}\label{LinftyL2}
1) For each $t>0$, the propagator is bounded as a map\footnote{Recall that $\psi \in e^{-N|c|}L^p(H^{-s}(\mathbb T),\mu_0)$ means there exists $f \in L^p(H^{-s}(\mathbb T),\mu_0)$ such that $\psi = e^{-N|c|}f$.}
\[ e^{-t{\bf H}}: L^\infty(H^{-s}(\T),\mu_0)\to e^{-N |c|} L^p(H^{-s}(\T),\mu_0)\]
for all $s>0$, $p\in [1,\infty)$ and $N\geq 0$, as well as for $(p,N)=(\infty,0)$. 
Here $c$ denotes the zero Fourier mode of a distribution in $ H^{-s}(\T)$.
In particular $e^{-t{\bf H}}:L^\infty(H^{-s}(\T),\mu_0)\to \mc{H}$ is bounded.
2) For each $t>0$ and $p\in [2,1+e^{2t})$ the propagator is bounded as a map
\begin{equation}\label{prop_L^2L^p}
e^{-t{\bf H}}: L^2(H^{-s}(\T),\mu_0)\to L^p(H^{-s}(\T),\mu_0).
\end{equation}
3) Let $q>1$ and $t>0$, then for all $N>0$ and all $p\in [1,q)$ the propagator is bounded as a map
\begin{equation}\label{prop_Weighted LpL^q}
e^{-t{\bf H}}: L^q(H^{-s}(\T),\mu_0)\to e^{-N|c|}L^p(H^{-s}(\T),\mu_0).
\end{equation}

\end{lemma}
\begin{proof} 1) The $p=\infty$ case is obvious. Let us write, using \eqref{propagatorT_t} and $|x|^{-\gamma Q}\geq 1$ when $|x|\leq 1$
\[ \begin{split}
|e^{-t{\bf H}}F(c,\varphi)|\leq &  \|F\|_{L^\infty} \E_{\varphi}\Big[e^{-\mu\int_{\A_t}|x|^{-\gamma Q}(e^{\gamma c}M_\gamma^{+,\C}(dx)+e^{-\gamma c}M^{-,\C}_\gamma(dx))} \Big]\\
\leq &   \|F\|_{L^\infty} \Big(\E_{\varphi}\Big[{\bf 1}_{c\geq 0}e^{-\mu e^{\gamma c} \int_{\A_t}M_\gamma^{+,\C}(dx)}+{\bf 1}_{c\leq 0}e^{-\mu e^{-\gamma c} \int_{\A_t}M_\gamma^{-,\C}(dx)} \Big]\Big).
\end{split}\]
Using H\"older inequality and a change of variable $e^{\gamma c}=y$, we get  for each $p\in [1,\infty)$ and $N\geq 0$
\begin{align*}
 &\int_0^\infty \int_{H^{-s}(\T)}e^{pN c}\Big(\E_{\varphi}\Big[e^{-\mu e^{\gamma c} \int_{\A_t}M_\gamma^{+,\C}(dx)}\Big]\Big)^pd\P_\T(\varphi)dc\\
& \leq  \int_0^\infty e^{pN c} \E\Big[e^{-p\mu e^{\gamma c} \int_{\A_t}M_\gamma^{+,\C}(dx)}\Big]dc= \frac{1}{\gamma} \int_1^\infty y^{\frac{N p}{\gamma}} \E\Big[e^{-p\mu y \int_{\A_t}M_\gamma^{+,\C}(dx)}\Big]\frac{dy}{y}\\
& \leq \frac{1}{\gamma} \int_0^\infty y^{\frac{N p}{\gamma}}  \E\Big[e^{-p\mu y \int_{\A_t}M_\gamma^{+,\C}(dx)}\Big]dy=\gamma^{-1}\Gamma(\tfrac{N p}{\gamma}+1)(p\mu)^{-1-\frac{Np}{\gamma}}
\E\Big[ \Big(M_\gamma^{+,\C}(\A_t)\Big)^{-1-\frac{N p}{\gamma}}\Big]<\infty,
\end{align*}
where in the final inequality we used the fact that the GMC has negative moments of any order $m<0$ (Lemma \ref{lem: GMC negative moments}). The same estimate holds on $(-\infty,0)$ with the  $M_\gamma^{-,\C}(dx)$ term and this proves the claim.

2) Next, we claim that $e^{-t{\bf H}^0}:\mc{H}\to L^p(H^{-s}(\T),\mu_0)$ for $p<1+e^{2t}$. 
We introduce the standard Hermite polynomials $(h_n)_{n \geq 0}$ on $\R$. Recall the following bound between the $L^2$ and $L^p$ norms of the Hermite polynomials for $p>2$: there is $c(p)>0$ such that for all $n\geq 0$
\begin{equation}\label{l2lp}
\|  h_n \|_p \leq c(p) n^{-1/4} (p-1)^{\frac{n}{2}} \|h_n \|_{2}.
\end{equation}
where the $L^p$ norm is with respect to the Gaussian measure $e^{-x^2/2}dx/\sqrt{2\pi}$. Consider  the normalised Hermite polynomials
for $\k=(k_1,\dots,k_N)\in \N^{N}$ and $\l=(l_1,\dots,l_{N'})\in \N^{N'}$ (with $N,N'\in \N$)
\[
\psi_{ \k, \l }(\varphi)= \prod_{n \geq 1}  \frac{h_{k_n}(x_n)}{\sqrt{k_n!}}\frac{h_{l_n}(y_n)}{\sqrt{l_n!}}, \quad  \quad  \quad    \|\psi_{ \k, \l }\|_{L^2(\Omega_\T)}=1.
\]
Each $F \in L^2(H^{-s}(\T))$ decomposes under the form 
$F(c+\varphi)=\sum_{\k,\l}F_{ \k, \l }(c)\psi_{\k,\l}(\varphi)$ with $\|F\|_2^2=\sum_{\k,\l}\|F_{\k,\l}\|_{L^2(\R)}^2$
and therefore
\begin{equation*}
(e^{-t {\bf H}^0}F)(c,\varphi)=\sum_{\k,\l}e^{-t(|\k|+|\l|)}F_{\k,\l}(c)\psi_{\k,\l}(\varphi)
\end{equation*}
where $| \k |= \sum_{n=1}^{\infty} n k_n$ (similarly for $\l$). Using \eqref{l2lp} and that $e^{t\pl_c^2/2}:L^2(\R)\to L^p(\R)$ is bounded for $p\geq 2$,
\begin{align*}
\| e^{-t {\bf H}^0} F \|_{L^p(H^{-s}(\T))} & \leq   \sum_{ \k, \l } e^{-t  (|\k|+|\l|)}  (\sqrt{p-1})^{\sum_n (k_n+l_n)}   \|  e^{t\pl_c^2/2}F_{ \k, \l } \|_{L^p(\R)}  \\
&  \leq C_p   \sum_{ \k, \l } (e^{-2t}(p-1))^{\frac{|\k|+|\l|}2}    \| F_{ \k, \l } \|^2_{L^2(\R)}. 
\end{align*}
This converges for $e^{-2t}(p-1)<1$ since $\sum_{ \k, \l }  (e^{-2t}(p-1))^{\frac{|\k|+|\l|}2}= \sum_{n=1}^{\infty} \sum_{m=1}^\infty p(n) p(m) (e^{-2t}(p-1))^{\frac{n+m}2 }$, 
where $p(n)$ is the number of partitions of $n$ and is bounded by $e^{C \sqrt{n}}$ for some $C>0$ when $n\to \infty$. 
Since $|e^{-t{\bf H}}F|\leq e^{t{\bf H}^0}(|F|)$ pointwise, we deduce that \eqref{prop_L^2L^p} holds.

3) Let $F\in L^q(H^{-s}(\T))$ for $q>1$ and $p<q$, let $r>1$ be defined by $p/q+1/r=1$ . Let us write 
\[W_t(c+X):=\int_{\A_t}|x|^{-\gamma Q}(e^{\gamma c}M_\gamma^{+,\C}+e^{-\gamma c}M^{-,\C}_\gamma(dx))\]
with $X=X_\D+P\varphi$ being the GFF on the unit disk as before.
We have using Jensen's inequality
\[\begin{split}
\|e^{N|c|}e^{-t{\bf H}}F\|_{L^p}^p \leq & \int_{\R} e^{pN|c|}\E[ \E_\varphi[|F(c+B_t+\varphi_t)|^pe^{-p\mu W_t(c+X)}]]dc\\
\leq & \Big(\int_{\R} \E[ \E_\varphi[|F(c+B_t+\varphi_t)|^{q}]dc \Big)^{p/q} \Big(\int_{\R}  e^{prN|c|}\E[\E_\varphi[e^{-pr\mu W_t(c+X)}]]dc\Big)^{1/r}\\
\leq & C_{p,q}\|e^{-t{\bf H}^0}F\|^{p}_{L^{q}}\leq  C'_{p,q}\|F\|^{p}_{L^{q}}
\end{split}
\]
for some constants $C_{p,q},C'_{p,q}$, where we used the bounds described in 1) above to estimate the term involving $W_t$.
\end{proof}

Notice that $e^{-t{\bf H}}$ also maps $L^p\to L^p$ for all $p\geq 2$ by interpolation between $L^\infty\to L^\infty$ and $L^2\to L^2$, and by duality (since it is symmetric), it also maps $L^p\to L^p$ for all $p\in [1,2]$.

Finally, let us describe the integral kernel of the propagator.
\begin{lemma}\label{Int_kernel_prop}
1) For $t>0$, the propagator $e^{-t{\bf H}}$ can be written under the form 
\begin{equation}\label{prop_kernel} 
e^{-t{\bf H}}F(c+\varphi)=\int_{H^{-s}(\T)}\mc{A}_{\A_t}(c,\varphi,c',\varphi')F(c'+\varphi')d\mu_{0}(c'+\varphi')
\end{equation}
for $s>0$ where the integral kernel is the measurable function for $\varphi,\varphi'\in H^{-s}_0(\T)$ and $c,c'\in \R$
\[\mc{A}_{\A_t}(c,\varphi,c',\varphi')=\mc{A}^0_{\A_t}(c,\varphi,c',\varphi')\E_{\varphi,\varphi'}\Big[e^{-\mu \sum_{\sigma=\pm 1}\int_{\A_t}|x|^{-2-\frac{\gamma^2}{2}}M_{\gamma}^{\sigma,\C}(\phi_{\A_t},dx)}\Big]\]
where $\phi_{\A_t}=X_{\A_t,D}+P_{\A_t}(c+\varphi,c'+\varphi')$ with $X_{\A_t,D}$ the Dirichlet GFF on $\A_t$ and $P_{\A_t}(c+\varphi,c'+\varphi')$ the harmonic function in $\A_t$ with boundary values $(c+\varphi,c'+\varphi')$ on the boundary $\T$ and $e^{-t}\T$, the expectation is taken with respect to $X_{\A_t,D}$,  and the function $\mc{A}^0_{\A_t}$ is defined by 
\begin{equation}\label{free_prop_formula} 
\mc{A}^0_{\A_t}(c,\varphi,c',\varphi')=\frac{1}{\sqrt{2\pi t}}(\prod_{n=1}^\infty(1-e^{-2tn})^{-1})e^{-\frac{(c-c')^2}{2t}} e^{-\sum_n
\frac{1}{4\sinh(tn)}(q_{tn}(x_n,x_n')+q_{tn}(y_n,y_n))}
\end{equation}
where $\varphi(\theta)=\sum_{n\not=0} \frac{x_n+{\rm sign}(n)iy_n}{2\sqrt{n}}e^{in\theta}$, $q_s(x,x')=x^2e^{-s}+{x'}^2e^{-s}-2xx'$ and 
$\varphi'(\theta)=\sum_{n\not=0} \frac{x_n'+{\rm sign}(n)iy_n'}{2\sqrt{n}}e^{in\theta}$.\\
2) The integral kernel $\mc{A}_{\A_t}$ belongs to $\mc{A}_{\A_t}\in L^2((H^{-s}(\T))^2,\mu_0^{\otimes 2})$, the operator $e^{-t{\bf H}}$ is Hilbert-Schmidt and thus compact on $\mc{H}$.
\end{lemma}
\begin{proof} First, by \cite[Proposition 6.1]{GKRV21_Segal} (see Equation (6.9) there), the integral kernel of $e^{-t{\bf H}^0}$ is given by $\mc{A}^0_{\A_t}$. Now, recall that $\phi(t,\theta):=c+B_t+\varphi_t(\theta)$ satisfies $\phi(0,\cdot)=c+\varphi$, then by the Feynman-Kac representation \eqref{propagatorT_t}, for each $F\in \mc{H}$ we have (with $Q=\gamma/2+2/\gamma$)
\[\begin{split} 
e^{-t{\bf H}}F(c+\varphi)=&\E[ F(c+B_t+\varphi_t)e^{-\mu\sum_{\sigma = \pm 1}e^{\sigma \gamma c}\int_{\A_t}|x|^{-\gamma Q}M_{\gamma}^{\sigma,\C}(dx)}]\\
& = \E[ F(c+B_t+\varphi_t)H_t(c+\varphi,c+ \varphi_t)]
\end{split}
\]
with $H_t$ defined as the conditional expectation
\[ H_t(c,\varphi,c',\varphi'):=\E\Big[ e^{-\mu\sum_{\sigma = \pm 1}e^{\sigma \gamma c}\int_{\A_t}|x|^{-\gamma Q}M_{\gamma}^{\sigma,\C}(dx) }\,\Big |\, \varphi_0=\varphi, B_t+\varphi_t=c'-c+\varphi'\Big].\]
Now, using the Markov property of the GFF, $X_\D|_{\A_t}=X_{\A_t,D}+P_{\A_t}(0,X_\D|_{e^{-t}\T})$ with $P_{\A_t}(f_1,f_2)$ the harmonic function on $\A_t$ with boundary values $f_1$ on $\T$ and $f_2$ on $e^{-t}\T$, therefore
we obtain
\[\E_{\varphi,\varphi'}\Big[e^{-\mu \sum_{\sigma=\pm 1}e^{\sigma \gamma c}\int_{\A_t}|x|^{-\gamma Q}M_{\gamma}^{\sigma,\C}(\phi_{\A_t},dx)}\Big]=H_t(c,\varphi,c',\varphi') \]
since $(X_\D+P\varphi)|_{\A_t}=X_{\A_t,D}+P(\varphi,c'-c+\varphi')=\phi_{\A_t}-c$. Thus \eqref{prop_kernel} is proved. 

Let us remark that for $s>0$, we have that
\begin{align*}
-\frac{1}{4\sinh(s)}q_{s}(x,x')\leq C_0 \Big( \frac{x^2}{2}+\frac{(x')^2}2 \Big)
\end{align*}
for every $C_0\in[\frac{1}{1+e^{s}},1)$. This can be established by considering the matrix associated to the difference of these quadratic forms and computing the eigenvalues. The optimal such $C_0$ for us will be the lower bound, $C_0= \frac{1}{1+e^s}$. Taking $s=tn$, we then compute for $p \in [2, 1+ e^{tn})$,  
\begin{align*}
\int_{\R^2} e^{-\frac{p}{4\sinh(tn)}(q_{tn}(x_n,x_n'))}e^{-\frac{x_n^2}{2}-\frac{{x'_n}^2}{2}}\frac{dx_ndx_n'}{2\pi}
&\leq
\int_{\R^2} e^{-\frac{x_n^2}{2(1-pC_0)^{-1}}-\frac{{x'_n}^2}{2(1-pC_0)^{-1}}}\frac{dx_ndx_n'}{2\pi}
=
\Big(1-\frac{p}{1+e^{tn}}\Big)^{-1}.
\end{align*}
This implies that for $p\in [2,1+e^{t})$, there is $C(t)>0$ depending only on $t>0$ and $p$ such that for all $c,c'\in \R$ 
\begin{equation}\label{boundL^2A0}
 \|\mc{A}^0_{\A_t}(c,\cdot,c',\cdot)\|_{L^p(\Omega_\T^2)}=\frac{e^{-\frac{(c-c')^2}{2t}}}{\sqrt{2\pi t}\prod_{n=1}^\infty(1-e^{-2tn})}\prod_{n=1}^\infty\Big(1-\frac{p}{1+e^{tn}}\Big)^{-1}\leq C(t)e^{-\frac{(c-c')^2}{2t}}.
\end{equation}

Let $\bar{c}=(c+c')/2$ and write $P_{\A_t}(c,c')=\bar{c}+P_{\A_t}(c-\bar{c},c'-\bar{c})$. By the maximum principle, 
$|P_{\A_t}(c-\bar{c},c'-\bar{c})|\leq |c-c'|$. 
Using that for each $\theta>0$, there is $C_\theta$ such that $e^{-x}\leq C_\theta x^{-\theta}$ for all $x>0$, 
we can bound the potential term by
\begin{align*}
& \E_{\varphi,\varphi'}\Big[e^{-\mu \sum_{\sigma=\pm 1}\int_{\A_t}|x|^{-2-\frac{\gamma^2}{2}}
M_{\gamma}^{\sigma,\C}(\phi_{\A_t},dx)}\Big]\\
& \leq 
C_\theta \mu^{-\theta }e^{-\gamma \theta \bar{c}} e^{\gamma \theta |c-c'|}\textbf{1}_{\bar{c}\geq 0}\E_{\varphi,\varphi'}\Big[\Big(\int_{\A_t}|x|^{-\gamma Q} M_{\gamma}^{+,\C}(X_{\A_t},dx)\Big)^{-\theta }\Big]\\
& \quad +C_\theta\mu^{-\theta}e^{\gamma \theta \bar{c}} e^{\gamma \theta|c-c'|}\textbf{1}_{\bar{c}\leq 0}\E_{\varphi,\varphi'}\Big[\Big(\int_{\A_t}|x|^{-\gamma Q} M_{\gamma}^{-,\C}(X_{\A_t},dx)\Big)^{-\theta}\Big],
\end{align*}
where $X_{\A_t}=X_{\A_t,D}+P_{\A_t}(\varphi,\varphi')$. Let $B\subset \A_t^\circ $ be a ball with boundary 
not intersecting $\pl \A_t$. The random variable $e^{\inf_{x\in B}P_{\A_t}(\varphi,\varphi')(x)}$ has finite moments of any order and by Lemma \ref{lem: GMC negative moments}  there is $C_{\theta,t}<\infty$ depending on $\theta,t$ such that 
\[\E \Big[\Big(\prod_{\sigma=\pm 1}\int_{B} M_{\gamma}^{\sigma,\C}(X_{\A_t,D},dx)\Big)^{-\theta}\Big]\leq C_{\theta,t}.\]
Consequently, for all $q\in [1,\infty)$ there is $C_{\theta,t,q}>0$ such that 
\[ \E\Big[ \Big(\E_{\varphi,\varphi'}\Big[e^{-\mu \sum_{\sigma=\pm 1}\int_{\A_t}|x|^{-\gamma Q}
M_{\gamma}^{\sigma,\C}(\phi_{\A_t},dx)}\Big]\Big)^q\Big]^{1/q} \leq C_{\theta,t,q}e^{-\gamma \theta |\bar{c}|}e^{\gamma \theta |c-c'|}.\]
Combining with \eqref{boundL^2A0} and using H\"older inequality with $1/q+1/p=1$ and $p\in (2,1+e^{t})$, we deduce that
there is $C_{\theta,t}>0$ and $C'_{\theta,t}>0$ such that
\[ \|\mc{A}_{\A_t}(c,\cdot,c',\cdot)\|_{L^2(\Omega_\T^2)}\leq C_{\theta,t} e^{-\gamma \theta |\bar{c}|}e^{\gamma \theta|c-c'|}e^{-\frac{(c-c')^2}{2t}}\leq  C'_{\theta,t} e^{-\gamma \theta  (|c+c'|+|c-c'|)}.\]
This function is in $L^2(\R^2,dc\otimes dc')$ thus we have proved that $\mc{A}_{\A_t}\in L^2((H^{-s}(\T))^2,\mu_0^{\otimes 2})$ and therefore the operator $e^{-t{\bf H}}$ is Hilbert-Schmidt and compact on $\mc{H}$.
\end{proof}
We obtain as a corollary that the resolvent of ${\bf H}$ is compact on $\mc{H}$:
\begin{corollary}\label{compact_resolvent}
The operator 
\[\bR_+=({\bf H}+1)^{-1}=\int_0^\infty e^{-t\H -t} dt\]
is compact on $\mc{H}$, ${\bf H}$ has discrete spectrum and an orthonormal basis of eigenfunctions $(\psi_j)_{j=0}^\infty$ with eigenvalues $\la_j\geq 0$.
\end{corollary}
\begin{proof}
By the spectral theorem, it follows that 
\[ \bR_+=\lim_{t\to 0}e^{-t{\bf H}}\bR_+\]
in operator norm on $\mc{H}$. Since $e^{-t{\bf H}}$ is compact on $\mc{H}$, the same is true for $e^{-t{\bf H}}\bR_+$. 
This proves that $\bR_+$ is compact since the space of compact operators is closed in $\mc{L}(\mc{H})$. The resolvent 
$\bR(\lambda)=({\bf H}-\lambda)^{-1}$ for $\lambda \notin \R_+$ 
satisfies the identity
\[ \bR(\lambda)(1-(1+\lambda)\bR_+)=\bR_+.\]
By the analytic Fredholm theorem we see that $(1-(1+\lambda)\bR_+)$ is invertible outside a discrete set of poles. Therefore 
$\bR(\lambda)$ extends to a meromorphic family in $\mc{L}(\mc{H})$ for $\lambda\in \C$. 
The spectrum of ${\bf H}$ is then discrete and given by the poles of $\bR(\lambda)$, all contained in $\R^+$. 
The eigenfunctions of ${\bf H}$ coincide with those of $\bR_+$, which form an orthonormal basis of $\mc{H}$.
\end{proof}

\subsection{Properties of eigenfunctions and proof of Theorem \ref{thm:spectrum_H}}

The semigroup $(e^{-t\H})_{t \geq 0}$ is positivity preserving in the sense that, for all $t > 0$ and $F\in\mc{H}$ such that $F\geq 0$ a.e. 
and $F$ is not identically $0$, we have $e^{-t\H}F\geq 0$. Indeed, this follows immediately from the Feynman-Kac representation \eqref{propagatorT_t}. We now show that the semigroup possesses a stronger property that ensures the \textit{strict positivity} of $e^{-t\H}F$ for $t>0$.

\begin{lemma} \label{lem: positivity improving}
The semigroup $(e^{-t\H})_{t>0}$ is positivity improving. That is, for all $t > 0$ and for all $F \in \mc{H}$ such that $F \geq 0$ $\mu_0$-a.e. and $F$ is not identically $0$, we have that, $\mu_0$-a.e., 
\begin{equation}
e^{-t\H}F (c+\varphi)>0.	
\end{equation}	
As a consequence, the multiplicity of the smallest eigenvalue $\lambda_0$ of ${\bf H}$ is equal to $1$ and the 
eigenfunction $\psi_0$ associated with $\lambda_0$ of ${\bf H}$ is positive.
\end{lemma}

\begin{proof}
Let $F\in \mc{H}$ such that $F\geq 0$ $\mu_0$-almost everywhere and is not identically 0 (almost everywhere). 
Thus there exists a Borel set $A \subset H^{-s}(\T)$ of strictly positive and finite measure such that $F(c+\varphi) > 0$ for $\mu_0$-a.e. $c+\varphi \in A$. 
Note that $\int_{A} F d\mu_0  > 0$ since $\int_{A} Fd\mu_0  \geq \int_{A_n}Fd\mu_0\geq \frac{\mu_0(A_n)}{n}$
 for all $n\geq 1$ with $A_n:=\{(c+\varphi)\in A\,|\, F(c+\varphi)\geq 1/n\}$,
and there is at least one $n\in \N$ such that $\mu_0(A_n)>0$ as otherwise $A$ would be a countable union of sets of measure $0$.

By Lemma \ref{Int_kernel_prop} we have for a.e. $c+\varphi$
\[\begin{split}
e^{-t{\bf H}}F(c+\varphi)=& \int_{H^{-s}(\T)}\mc{A}_{\A_t}(c,\varphi,c',\varphi')F(c'+\varphi')d\mu_{0}(c'+\varphi')\\
\geq &  \int_{A}\mc{A}_{\A_t}(c,\varphi,c',\varphi')F(c'+\varphi')d\mu_{0}(c'+\varphi').
\end{split}\]
Since for a.e. $c+\varphi$, $\mc{A}^0_{\A_t}(c,\varphi,\cdot )>0$ a.e. on $A$, 
it suffices to show that for a.e. $c+\varphi$,
\[ \int_{A'}\E_{\varphi}\Big[e^{-\mu \sum_{\sigma=\pm 1}\int_{\A_t}|x|^{-\gamma Q}M_{\gamma}^{\sigma,\C}(\phi_{\A_t},dx)}\Big] d\mu_0(c'+\varphi')>0
\]  
for some $A'\subset A$ with $\mu_0(A')>0$. Assume by contradiction that the integral above vanishes $0$ for $c+\varphi$ in a set $B$ of positive and finite measure and $A'=A$. Then we would have 
\[  \int_{\R^2}\E[1_{c'+\varphi'\in A}e^{-\mu \sum_{\sigma=\pm 1}\int_{\A_t}|x|^{-\gamma Q}M_{\gamma}^{\sigma,\C}(\phi_{\A_t},dx)}]dcdc'=0,\]
which would mean that for a.e. $c+\varphi\in B$ and  $c'+\varphi'\in A$
\[\sum_{\sigma=\pm 1}\int_{\A_t}|x|^{-\gamma Q}M_{\gamma}^{\sigma,\C}(\phi_{\A_t},dx)=\infty.\]
Note that there is an interval $I_n:=[-n,n]$ such that $B_n:=\{c+\varphi\in B\,|\, c\in I_n\}$ has positive $\mu_0$-measure
and  $A_n:=\{c+\varphi\in A\,|\, c\in I_n\}$ has positive $\mu_0$-measure.
For all $p>0$ we have 
\begin{align*}
& \infty=\int_{\R^2}\E\Big[1_{B_n}(c+\varphi)1_{A_n}(c'+\varphi')\Big(\sum_{\sigma=\pm 1}\int_{\A_t}|x|^{-\gamma Q}M_{\gamma}^{\sigma,\C}(\phi_{\A_t},dx)\Big)^p\Big]dc dc'\\
& \quad \leq \int_{\R}
\mathbb{E}\left[ 1_{B_n}(c''+X|_{\T}) 1_{A_n}(c''+X|_{e^{-t}\T})\left(\sum_{\sigma = \pm 1} \int_{\mathbb{A}_t}|x|^{-\gamma Q} e^{\sigma \gamma c''} M_\gamma^{\sigma,\mathbb{C}}(X,dx) \right)^p\right]dc'' \\
& \quad \leq \int_{I_n}
\mathbb{E}\left[ \left(\sum_{\sigma = \pm 1} \int_{\mathbb{A}_t}|x|^{-\gamma Q} e^{\sigma \gamma c''} M_\gamma^{\sigma,\mathbb{C}}(X,dx) \right)^p\right]dc'', 
\end{align*}
where $X=X_\D+P\varphi$ is the plane GFF restricted to $\D$, which satisfies the equality in law $c''+X=\Phi_{\A_t}$ under the condition $c''+X|_{\T}=c+\varphi$ and $c''+X|_{e^{-t}\T}=c'+\varphi'$, in particular $c''=c$ and $c+B_t=c'$.
By Lemma \ref{lem: GMC positive moments}, there exists $p>0$ such that for all $c\in I_n$
\[ \E \Big[\left(\sum_{\sigma = \pm 1} \int_{\mathbb{A}_t}|x|^{-\gamma Q} e^{\sigma \gamma c} M_\gamma^{\sigma,\mathbb{C}}(X,dx) \right)^p\Big]<\infty,\] 
thus we obtain a contradiction.

The statement about simplicity of smallest eigenvalue and the positivity of the ground state follows from the positivity improving property of $e^{-t{\bf H}}$ by applying \cite[Theorem XIII.44]{Reed-Simon4}.
\end{proof}

We next show that the eigenfunctions belong to  weighted $L^p$ spaces:
\begin{lemma}\label{lem:eigenfunction_decay}
For $s>0$, the eigenfunctions $\psi_j$ of ${\bf H}$ belong to $e^{-N|c|}L^p(H^{-s}(\T),\mu_0)$ for all $p<\infty$ and all $N\geq 0$.
\end{lemma}
\begin{proof} Let $\lambda_j$ be the eigenvalue of ${\bf H}$ associated to $\psi_j\in \mc{H}$. By Lemma \ref{LinftyL2}, point 2), for all $t>0$ we have $\psi_j=e^{t\lambda_j} e^{-t{\bf H}}\psi_j\in L^p(H^{-s}(\T))$ for all $p<1+e^{2t}$, and by point 3) of Lemma \ref{LinftyL2} we then get the result since $t>0$ can be chosen arbitrarily large.
\end{proof}

We now combine the results of this section to prove Theorem \ref{thm:spectrum_H}.

\begin{proof}[Proof of Theorem \ref{thm:spectrum_H}]
The existence of $\mathbf{H}$ and the Feynman-Kac formula in (i) follow from Lemmas \ref{TtFeynmannKac} and \ref{lem:Q_equal_Qstar}. The properties of the spectrum in (ii) and (iv) follow from Lemma \ref{lem: positivity improving}. The property of the eigenfunctions in (iii) follows from Lemma \ref{lem:eigenfunction_decay}.
\end{proof}

\section{Construction of the Sinh-Gordon model and one-point function} \label{sec: path integral}

The goal of this section is to prove Theorems \ref{thm: sinh and one pt} and \ref{thm: sinh correlations}, which address the construction of the model and its vertex correlations respectively. 
Theorem \ref{thm: sinh and one pt} is proved in Sections \ref{subsec: construction sinh 2pi}-\ref{subsec: construction sinh R}. Theorem \ref{thm: sinh correlations} is proved in Section \ref{subsec:_vertex_correlations}. Unless stated otherwise, we fix $\mu > 0$ and $\gamma \in (0,2)$ and drop them from the notation.

\subsection{State space for the models} \label{subsec: state space}

Let $R>0$ and $s>0$ fixed. Denote by $C(\mathbb{R}, H^{-s}(\T_R))$ the space of continuous functions from $\mathbb{R}$ into $H^{-s}(\T_R)$. We are going to define the Sinh-Gordon models on this space. 
Note that, for some fixed $s>0$, there is a canonical continuous injection
\[
C(\mathbb{R}, H^{-s}(\mathbb{T})) \hookrightarrow \mc{D}'(\mathcal{C}_R)
\]
into the space of distributions. Via pushforward of these injections, we obtain a measure on $D'(\mathcal{C}_R)$. In the sequel we shall identify these measures and ignore this subtlety. In order to define these measures, we identify a sufficiently rich ring of observables, corresponding to compactly supported functions, on which we can construct a candidate premeasure by probabilistic techniques. We then extend them to a full probability measure by Carath\'eodory extension and controlling the normalization. We now make this precise. 

Given $I \subset \mathbb{R}$ an interval and $\phi \in C(\mathbb{R}, H^{-s}(\T_R))$, denote by $\phi|_I$ the restricted function in $C(I, H^{-s}(\T_R))$. 
We say that $F:C(\mathbb{R}, H^{-s}(\T_R)) \rightarrow \mathbb{R}$ is $I$-measurable if there exists 
$F^I : C(I, H^{-s}(\T_R)) \rightarrow \mathbb{R}$ such that
\begin{align*}
F(\phi) 
= 
F^I(\phi|_I), \qquad \forall 	\phi \in C(\mathbb{R}, H^{-s}(\T_R)).
\end{align*} 
Henceforth, we identify $F$ and $F^I$. We say that $F:C(\mathbb{R}, H^{-s}(\T_R)) \rightarrow \mathbb{R}$ is of compact support if there exists a closed interval $I \subset\mathbb{R}$ of finite length such that $F$ is $I$-measurable. We then define the support of $F$ to be the smallest closed interval $I$ such that $F$ is $I$-measurable. Let $\mathcal{F}^R$ denote the ring of bounded measurable functions $F:C(\mathbb{R},H^{-s}(\T_R))\rightarrow \mathbb{R}$ that are of compact support. Note that $\mathcal{F}^R$ generates the full $\sigma$-algebra in light of the topology of compact convergence on $C(\mathbb{R},H^{-s}(\T_R))$. We shall write $\mc{F}$ for $\mc{F}^{R=1}$.

\subsubsection{Actions of translations, scaling, and a conformal map} 

In the construction of the measures, to simplify formulas it will be convenient to act on functions in $\mathcal{F}^R$ by three mappings: translations in the $t$ variable, a scaling variable between $R>0$ and $R=1$, and a conformal mapping to the punctured disk.   

Given $T>0$, let $\tau_T$ denote the time-translation map $t \mapsto t+T$ which acts on $C(\mathbb{R}, H^{-s}(\T_R))$ via $\phi \mapsto \phi(T+\cdot)$. We extend its action to functions $F:C(\mathbb{R}, H^{-s}(\T_R)) \rightarrow \mathbb{R}$ via  $\tau_T^*F(\phi) := F(\tau_T \phi)$ for all $\phi \in C(\mathbb{R}, H^{-s}(\T_R))$. In particular, for functions of compact support, this shifts the support by $+T$.

Denote by $\rho_R:\mathcal{C}_R \rightarrow \mathcal{C}_1$ the scaling map $\rho_R:(t,\theta) \mapsto (t/R, \theta/R)$. 
We extend its action to any function $F:C(\mathbb{R}, H^{-s}(\T_R)) \rightarrow \mathbb{R}$ by 
$\rho_R^*F: C(\mathbb{R}, H^{-s}(\T)) \rightarrow \mathbb{R}$ being defined by 
$\rho_R^*F(\phi):=F(\phi\circ \rho_R)$ (here we view $\phi \in C(\mathbb{R}, H^{-s}(\T))$ as a distribution of 
the variable $(t,\theta)$).

\subsection{Sinh-Gordon model for $R=1$.} \label{subsec: construction sinh 2pi}

We now construct the Sinh-Gordon model on $\mathcal{C}_1$. We will use the notation 
\[\mc{C}_{[t_1,t_2]}:= [t_1,t_2]\times \T\] 
(in particular $\mc{C}^+_{1,T}=\mc{C}_{[0,T]}$).
We first introduce our approximations which will depend on a choice of initial configuration for the underlying GFF (viewed as a stochastic process).

\begin{definition}
Let $T>0$. Denote by $\langle \cdot \rangle_{\mc{C}_{1,T}}$ the probability measure on $C(\mathbb{R},H^{-s}(\T))$ with expectation values given, for any bounded measurable $F$ with support in $[-T,T]$, by
\[
\langle F \rangle_{\mc{C}_{1,T}}
:=
\frac{1}{Z_{\mc{C}_{1,T}}} \int_\R \E[F (c+B_{\bullet+T}+\varphi_{\bullet+T}) \,   e^{- \mu\sum_{\sigma = \pm 1} 
e^{\sigma\gamma c} M_\gamma^\sigma(\mc{C}_{[0,2T]})}]dc.
\]
Above, $B_t$ is the standard Brownian motion, $\varphi_t$ is the process \eqref{varphit_def},  $c+B_{\bullet+T}+\varphi_{\bullet+T}$ denotes the functions 
$t\mapsto  c+B_{t+T}+\varphi_{t+T}$ in $C(\R,H^{-s}(\T))$, $M_\gamma^\pm$ is the GMC measure depending on $\varphi$ defined in Proposition \ref{prop: GMC construction},  and 
\[\begin{split}
Z_{\mc{C}_{1,T}} & = \int_\R \mathbb{E}\Big[e^{- \mu\sum_{\sigma = \pm 1} 
e^{\sigma\gamma c} M_\gamma^\sigma(\mc{C}_{[0,2T]})}\Big]dc =\int_{H^{-s}(\T)}(e^{-2T\H}1)d\mu_0
\end{split}
\] 
is the partition function.
\end{definition}

\begin{remark}
We have used shift-invariance of the process $(B,\varphi)$ to translate the problem onto the interval $[0,2T]$ from $[-T,T]$. The translation simplifies some of the intermediary formulas. 
\end{remark}

We now construct the limits of the averages $\langle F \rangle_{\mc{C}_{1,T}}$ as $T \rightarrow \infty$ and show that the limiting value is given by an explicit formula. In fact we show a stronger statement, namely a pointwise convergence $dc\otimes \mathbb{P}_\mathbb{T}$-almost everywhere along subsequences $T_{n_k}$. 
These constants are given in the following definition. 

\begin{definition}\label{def:<F>}
Let $F \in \mathcal{F}:=\mc{F}^1$ with support contained in $[t_1,t_2]$ where $-\infty < t_1 < t_2 < \infty$ and set
\[\begin{split}
\cjg F\cjd_{\mc{C}}=	 e^{\lambda_0(t_2-t_1)} \Big\langle 	\mathbb{E}_{\varphi} \Big[ F(c+B_{\bullet+t_1}+
\varphi_{\bullet+t_1})\psi_0(c+B_{(t_2-t_1)}+
\varphi_{t_2-t_1}) e^{- \mu\sum_{\sigma = \pm 1} 
e^{\sigma\gamma c} M_\gamma^\sigma(\mc{C}_{[0,t_2-t_1]})}\big], \psi_0 \Big \rangle_{\mc{H}},\end{split}\]
where $\psi_0$ is the eigenfunction with norm $\|\psi_0\|_{\mc{H}}=1$ associated the smallest eigenvalue $\la_0\geq 0$ and is strictly positive $\mu_0$-almost everywhere.
\end{definition}

Using estimates from Lemma \ref{LinftyL2} one can show that $\cjg F\cjd_{\mc{C}}$ is finite: indeed, bounding $|F|\leq C$, it suffices to estimate 
\begin{equation}\label{bound_EF}
 \cjg F\cjd_{\mc{C}} \leq  \Big(\sup_{\phi\in C^0(\R,H^{-s}(\T))}|F(\phi)|\Big)e^{\lambda_0(t_2-t_1)}\cjg e^{-(t_2-t_1){\bf H}}\psi_0,\psi_0\cjd_{\mc{H}}=\sup_{\phi\in C^0(\R,H^{-s}(\T))}|F(\phi)|. 
 \end{equation}
We now turn to the convergence of the averages.

\begin{proposition} \label{prop: sinh expectation}
Let $F \in \mathcal{F}$ with support $[t_1,t_2]$ where $-\infty < t_1 \leq t_2 < \infty$. Then 
\[\lim_{T\to \infty} \langle F \rangle_{\mc{C}_{1,T}}
=\cjg F\cjd_{\mc{C}}.\]
Moreover,  for each sequence $T_n\to \infty$, there is a subsequence $T_{n_k}$ such for $\mu_0$-almost every $c+\varphi\in H^{-s}(\T)$ (with $s>0$),
\begin{equation}\label{ae_convergence}
\lim_{k\to \infty} \frac{1}{(e^{-T_{n_k}{\bf H}}1)(c+\varphi)}  \E_\varphi[F (c+B_{\bullet+T_{n_k}}+\varphi_{\bullet+T_{n_k}}) \,   e^{- \mu\sum_{\sigma = \pm 1} 
e^{\sigma\gamma c} M_\gamma^\sigma(\mc{C}_{[0,2T_{n_k}]})}]
=\cjg F\cjd_{\mc{C}}.	
\end{equation}
Finally, the first eigenfunction and eigenvalues can be obtained by the expressions 
\[ \la_0=-\frac{1}{2}\lim_{T\to +\infty}\frac{1}{T}\log Z_{\mc{C}_{1,T}}, \quad \Big(\int \psi_0d\mu_0\Big) \psi_0= \lim_{T\to \infty}e^{\la_0T}e^{-T{\bf H}}1,\]
where the second limit is with the topology of $\mc{H}$.
\end{proposition}

\begin{proof}
By the shift-invariance we may without loss of generality consider $F \in \mathcal{F}$ with support in $[0,t_0]$ for $t_0>0$. 

In order to analyse the partition function, let us make the following observations.
By Lemma \ref{LinftyL2} we have $e^{-\eps\H}1\in \mc{H}$ that for any $\eps >0$, and 
\begin{align*}
e^{-2T\H}1
=
\sum_{j=0}^\infty e^{-\lambda_j (2T-\eps)} \langle e^{-\eps\H}1,\psi_j \rangle_{\mc{H}} \psi_j.
\end{align*}
with $C_j(\eps):= \langle e^{-\eps\H}1,\psi_j \rangle_{\mc{H}}$ satisfying $\sum_j |C_j(\eps)|^2<\infty$.
In particular, we deduce that for $\eps,\delta>0$ small enough so that $\la_0+\delta<\la_1-2\eps$
\[ e^{-2T\H}1=e^{-\lambda_0 (2T-\eps)} \langle e^{-\eps\H}1,\psi_0 \rangle_{\mc{H}} \psi_0+ \mc{O}_{\mc{H}}(e^{-2(\la_0+\delta)T}).\]
Since $e^{-t{\bf H}}$ is Hilbert-Schmidt for all $t>0$, we see that$\sum_{j=1}^\infty e^{-\la_jt}<\infty$ for all $t>0$.
Using  that there is $C_\eps>0$
\[\|\psi_j\|_{L^1(H^{-s}(\T),\mu_0)}=\|e^{\eps \la_j}e^{-\eps{\bf H}}\psi_j \|_{L^1(H^{-s}(\T),\mu_0)}\leq C_\eps e^{\eps \la_j} \]
we also get an $L^1(H^{-s}(\T),\mu_0)$ estimate using the Cauchy-Schwarz inequality:
\begin{equation}\label{L^1remainder}
\begin{split} 
\|e^{-2T\H}1-&e^{-\lambda_0 (2T-\eps)} \langle e^{-\eps\H}1,\psi_0 \rangle_{\mc{H}} \psi_0\|_{L^1(H^{-s}(\T),\mu_0)}
\\ &\leq  C_\eps \Big(\sum_{j\geq 1}e^{-2\la_j(T-2\eps)}\Big)^{1/2}\Big(\sum_{j\geq 1} |C_j(\eps)|^2\Big)^{1/2}
\leq  C'_\eps e^{-2(\la_0+\delta)T}.
\end{split}
\end{equation}
Now, we also have 
\[\begin{split}
\langle e^{-\eps\H}1,\psi_0 \rangle_{\mc{H}}=&\lim_{N\to \infty}\langle e^{-\eps\H}1_{[-N,N]}(c),
\psi_0 \rangle_{\mc{H}}
= \lim_{N\to \infty}\langle 1_{[-N,N]}(c),
e^{-\eps\H}\psi_0 \rangle_{\mc{H}}\\
 =& e^{-\eps \la_0}\lim_{N\to \infty}\langle 1_{[-N,N]}(c),
\psi_0 \rangle_{\mc{H}}=e^{-\eps \la_0}\int_\R \E_{\mathbb T}[\psi_0] dc,
\end{split}\]
where we used Lemma \ref{LinftyL2} in the last identity. 
This gives 
\[e^{-2T\H}1=e^{-2\lambda_0 T}\Big(\int \psi_0 d\mu_0\Big) \psi_0 +\mc{O}_{\mc{H}}(e^{-2(\la_0+\delta)T}),\]
and thus for any sequence $T_n\to \infty$ there is a subsequence $T_{n_k}\to \infty$ such that
\[\lim_{k\to \infty }e^{2\lambda_0T_{n_k}}(e^{-2T_{n_k}\H}1)(c+\varphi)= \int \psi_0 d\mu_0\] 
for $\mu_0$-almost all $(c+\varphi)$. Using \eqref{L^1remainder} we also have 
\[ \lim_{T\to +\infty}e^{2\lambda_0 T}Z_{\mc{C}_{1,T}}= \Big(\int \psi_0 d\mu_0\Big)^2. \]

We now estimate $\langle F \rangle_{\mc{C}_{1,T}}$ for $T>t_0$. We write $\int_{H^{-s}(\T)}W_T(c+\varphi)d\mu_0(c+\varphi):=Z_{\mc{C}_{1,T}}\langle F \rangle_{\mc{C}_{1,T}}$ where, using the Markov property and translation invariance of the GFF, the function $W_T$ can be  written as 
\begin{align}\label{eq: constr 1}
W_T(c+\varphi)=&
\E_{\varphi}\Big[F((c+B_{\bullet+T}+\varphi_{\bullet+T})|_{t\in [0,t_0]}) \,  e^{-\mu \sum_{\sigma = \pm 1}e^{\sigma\gamma c} M_\gamma^\sigma(\mc{C}_{[0,2T]})} ]\\
=& \int_{(H^{-s}(\T))^2}\mc{A}_{\A_T}(c+\varphi,c'+\varphi')\mc{B}_T(c'+\varphi',c''+\varphi'')G_T(c'',\varphi'')d\mu_0(c'+\varphi')d\mu_0(c''+\varphi'')\\
\end{align}
where 
\[ \mc{A}_{\A_T}(c+\varphi,c'+\varphi'):=\E_\varphi [e^{-\mu \sum_{\sigma = \pm 1}e^{\sigma\gamma c} M_\gamma^\sigma(\mc{C}_{[0,T]})}\,|\, B_T+\varphi_T=c'-c+\varphi']\]
is the integral kernel of $e^{-T{\bf H}}$, 
\[ \begin{split}
G_T(c'',\varphi''):=&\E[e^{-\mu \sum_{\sigma = \pm 1}e^{\sigma\gamma c''} M_\gamma^\sigma(\mc{C}_{[T+t_0,2T]})}\,|\, B_{T+t_0}+\varphi_{T+t_0}=\varphi'']\\
=& (e^{-(T-t_0)\H}1)(c'',\varphi'')
\end{split}\]
and 
\begin{align*}
& \mc{B}_T(c'+\varphi',c''+\varphi'')\\
& :=\E\left[ F((c'+B_{\bullet+T}+\varphi_{\bullet+T})|_{t\in [0,t_0]})
e^{-\mu \sum_{\sigma = \pm 1}e^{\sigma\gamma c'} M_\gamma^\sigma(\mc{C}_{[T,T+t_0]})}\,\Bigg|\, 
\begin{array}{l}
B_{T+t_0}+\varphi_{T+t_0}=c''-c'+\varphi'',\\
B_T+\varphi_T=\varphi' \end{array}\right]\\
&= \mathbb E\left[ F((c'+B_\bullet+\varphi_{\bullet})|_{t\in [0,t_0]})
e^{-\mu \sum_{\sigma = \pm 1}e^{\sigma\gamma c'} M_\gamma^\sigma(\mc{C}_{[0,t_0]})}\,\Bigg|\, 
\begin{array}{l}
B_{t_0}+\varphi_{t_0}=c''-c'+\varphi'',\\
\varphi_{0}=\varphi' \end{array}\right].
\end{align*}
In other words, we have written 
\begin{equation}\label{rewrittingWT}
W_T=e^{-T{\bf H}}\mc{B}_Te^{-(T-t_0){\bf H}}1,
\end{equation}
where $\mc{B}_T$ is here viewed as linear operator with integral kernel $\mc{B}_T(c+\varphi,c'+\varphi')$. This can be interpreted as the gluing property from the Segal axioms picture, see \cite[Proposition 5.1]{GKRV21_Segal} (see the Remark below).
We see that $\mc{B}_T$ is independent of $T$ and will thus be denoted $\mc{B}$. Since $F$ is bounded, we can bound $\mc{B}$ by the integral kernel of the propagator
\[ |\mc{B}(c'+\varphi',c''+\varphi'')|\leq \|F\|_{L^\infty}\mc{A}_{\A_{t_0}}(c'+\varphi',c''+\varphi'').\]
Thus $\mc{B}(\bullet,\bullet)\in L^2(H^{-s}(\T)^2,d\mu_0^2)$ and we can view $\mc{B}_T$ as a Hilbert-Schmidt operator on $\mc{H}$, whose operator norm satisfies $\|\mc{B}\|_{\mc{L}(\mc{H})}\leq C_{t_0}$ for some constant $C_{t_0}$ independent of $T$.
We can thus rewrite 
\[ W_T(c+\varphi)= (e^{-T\H} \mc{B}e^{-(T-t_0)\H}1)(c,\varphi).\]
Using Lemma \ref{LinftyL2} and the bound $\|\mc{B}\|_{\mc{L}(\mc{H})}\leq C_{t_0}$,  we see that there is $\delta>0$ such that for each $\eps>0$ and $T\gg t_0$ large 
\begin{align}
e^{-(T-t_0)\H}1\in \mc{H}, &\qquad  e^{-(T-t_0)\H}1= e^{-\la_0(T-t_0-\eps)}\cjg e^{-\eps \H}1,\psi_0\cjd_{\mc{H}}\psi_0+
\mc{O}_{\mc{H}}(e^{-(\la_0+\delta)T}),	
\\
\mc{B}e^{-(T-t_0)\H}1\in \mc{H}, &\qquad \mc{B}e^{-(T-t_0)\H}1= e^{-\la_0(T-t_0-\eps)}\cjg e^{-\eps \H}1,\psi_0\cjd_{\mc{H}} \mc{B}\psi_0+\mc{O}_{\mc{H}}(e^{-(\la_0+\delta)T}),
\\
W_T\in \mc{H}, &\qquad W_T=e^{-\la_0(2T-t_0)}\Big( \int \psi_0d\mu_0\Big) \cjg \mc{B}\psi_0,\psi_0\cjd_{\mc{H}}\psi_0+\mc{O}_{\mc{H}}(e^{-(2\la_0+\delta)T}).\label{W_Tremainder}
\end{align}
Arguing as above for $Z_{\mc{C}_{1,T}}$, this means that, up extracting another subsequence, we can assume that the sequence $T_{n_k}$ is such that  $\mu_0$-almost everywhere 
\[ \lim_{k\to \infty}e^{2\la_0T_{n_k}}W_{T_{n_k}}(c+\varphi)=e^{\la_0t_0}\Big( \int \psi_0 d\mu_0\Big)\cjg \mc{B}\psi_0,\psi_0\cjd_{\mc{H}}\psi_0(c+\varphi).\]
Combining with the asymptotic of $Z_{\mc{C}_{1,T}}$, and using that $\mu_0$-almost everywhere the ground state  $\psi_0(c+\varphi)$ is positive, we obtain that for $\mu_0$-almost every $c+\varphi$, 
\[\lim_{k\to \infty} \frac{1}{(e^{-T_{n_k}{\bf H}}1)(c+\varphi)}  \E_\varphi[F (c+B_{\bullet+T_{n_k}}+\varphi_{\bullet+T_{n_k}}) \,   e^{- \mu\sum_{\sigma = \pm 1} 
e^{\sigma\gamma c} M_\gamma^\sigma(\mc{C}_{[0,2T_{n_k}]})}]
=e^{\la_0t_0} \cjg \mc{B}\psi_0,\psi_0\cjd_{\mc{H}}.\]
To obtain \eqref{ae_convergence}, it finally suffices to observe that the right-hand side is simply $\cjg F\cjd_{\mc{C}}$ as defined by \eqref{def:<F>}.
In the expansion of $W_T$ in \eqref{W_Tremainder}, the remainder is in $\mc{H}$ norm, but applying the same argument as for $Z_{\mc{C}_{1,T}}$ in \eqref{L^1remainder}, we obtain an $L^1$ remainder estimate
\[ W_T=e^{-\la_0(2T-t_0)}( \int \psi_0d\mu_0) \cjg \mc{B}\psi_0,\psi_0\cjd_{\mc{H}}\psi_0+\mc{O}_{L^1(H^{-s}(\T),\mu_0)}(e^{-(2\la_0+\delta)T}),\]
from which we deduce that 
\[\lim_{T\to \infty}\cjg F\cjd_{\mc{C}_{1,T}}=\cjg F\cjd_{\mc{C}}.\qedhere\]
\end{proof}

\textbf{Remark about the path integral.} We make the following remark, which relates to the discussion about the path integral on the cylinder $\mc{C}_{[-T,T]}$ in Section \ref{s:Path_Integral}. Following the method introduced in \cite[Definition 4.2]{GKRV21_Segal}, we can define for $t_1<t_2$ a Segal amplitude associated to a continuous function $F$  on $H^{-s}(\mc{C}_{[t_1,t_2]})$ and $c+\varphi,c'+\varphi'\in H^{-s}(\T)$ for $s>0$ as 
\[\mc{A}_{\mc{C}_{[t_1,t_2]}}(F)(c+\varphi,c'+\varphi'):= \mc{A}^0_{\A_{t_2-t_1}}(c+\varphi,c'+\varphi')\E_{\varphi,\varphi'}\Big[F(\phi) \,  e^{-\mu \sum_{\sigma = \pm 1}e^{\sigma\gamma c} M_\gamma^\sigma(\mc{C}_{[t_1,t_2]})} \Big],\]
where $\mc{A}^0_{\A_{t}}$ is the free propagator kernel from \eqref{free_prop_formula}, 
$\phi:=X_{\mc{C}_{[t_1,t_2]},D}+P(c+\varphi,c'+\varphi')$  with $X_{\mc{C}_{[t_1,t_2]},D}$ the Gaussian Free Field with Dirichlet condition at $t=t_1,t=t_2$ and $P(c+\varphi,c'+\varphi')$ is the harmonic function on $\mc{C}_{[t_1,t_2]}$ with boundary values $c+\varphi$ at $t=t_1$ and $c'+\varphi'$ at $t=t_2$ (the expectation $\E_{\varphi,\varphi'}$ is taken over the random variable $X_{\mc{C}_{[t_1,t_2]},D}$). In addition, the GMC measure $M^\sigma_\gamma$ is defined using the field $\phi$. Using the same reasoning as in the proof of Proposition \ref{prop: sinh expectation} (in particular \eqref{rewrittingWT}), we see that if $F \in \mathcal{F}$ has support in $[0,t_0]$, then the function $W_T(c+\varphi)$ of \eqref{eq: constr 1} (satisfying 
$\cjg F\cjd_{\mc{C}_{1,T}}=\int W_T(c+\varphi)d\mu_0(c+\varphi)$) can also be written in the form
\[ W_T(c+\varphi)= \int_{H^{-s}(\T)} \mc{A}_{\mc{C}_{[-T,T]}}(F)(c+\varphi,c'+\varphi')d\mu_0(c'+\varphi')=(\mc{A}_{\mc{C}_{[-T,T]}}(F)1)(c+\varphi),\]
where we now view $\mc{A}_{\mc{C}_{[-T,T]}}(F):L^1(H^{-s}(\T),\mu_0)\to \mc{H}$ as the integral (bounded) operator associated with this kernel. In fact, the following identity holds at the level of operators 
$\mc{A}_{\mc{C}_{[-T,T]}}(F)=e^{-T{\bf H}}\mc{A}_{\mc{C}_{[0,t_0]}}e^{-(T-t_0){\bf H}}$.
In other words, $\mc{A}_{\mc{C}_{[-T,T]}}(F)(c+\varphi,c'+\varphi')$, $W_T(c+\varphi)$ and $Z_{\mc{C}_{1,T}}\langle F \rangle_{\mc{C}_{1,T}}$ represent the path integrals 
\begin{align}
& \int_{H^{-s}_{c+\varphi,c'+\varphi'}(\mc{C}_{[-T,T]})} F(\phi) e^{-\frac{1}{4\pi} \int_{\mathcal{C}_{[-T,T]}}(|\nabla \phi|^2+\mu\cosh(\gamma \phi))dx} D\phi:=\mc{A}_{\mc{C}_{[-T,T]}}(F)(c+\varphi,c'+\varphi'),\label{double_partialPI}\\
& \int_{H^{-s}_{c+\varphi}(\mc{C}_{[-T,T]})} F(\phi) e^{-\frac{1}{4\pi} \int_{\mathcal{C}_{[-T,T]}}(|\nabla \phi|^2+\mu\cosh(\gamma \phi))dx} D\phi:=W_T(c+\varphi),\label{partialPI}\\
& \int_{H^{-s}(\mc{C}_{[-T,T]})} F(\phi) e^{-\frac{1}{4\pi} \int_{\mathcal{C}_{[-T,T]}}(|\nabla \phi|^2+\mu\cosh(\gamma \phi))dx} D\phi:= Z_{\mc{C}_{1,T}}\langle F \rangle_{\mc{C}_{1,T}}, \label{fullPI}
\end{align}
where we used the notation of Section \ref{s:Path_Integral} for $H^{-s}_{c+\varphi,c'+\varphi'}(\mc{C}_{[-T,T]})$ and 
$H^{-s}_{c+\varphi}(\mc{C}_{[-T,T]})$. Proposition \ref{prop: sinh expectation} then shows that as $T\to \infty$ the full path integral \eqref{fullPI} and the partial path integral \eqref{partialPI} with one fixed boundary value have the same asymptotic behaviour.
Using that $\mc{A}_{\mc{C}_{[0,t_0]}}(F)(\cdot,\cdot)\in L^2(H^{-s}(\T)^2,\mu_0^{\otimes 2})$ (by the same proof as in Lemma \ref{Int_kernel_prop}), we have 
\[(e^{-T{\bf H}}\otimes e^{-(T-t_0){\bf H}})\mc{A}_{\mc{C}_{[0,t_0]}}= e^{-\la_0 (2T-t_0)} \cjg \mc{A}_{\mc{C}_{[0,t_0]}}(F),\psi_0 \otimes \psi_0\cjd_{\mc{H}^{\otimes 2}} \psi_0\otimes \psi_0  +\mc{O}_{\mc{H}^{\otimes 2}}(e^{-(\la_0+\delta) (2T-t_0)} ).\] 
In addition, as in the proof of Proposition \ref{prop: sinh expectation}, the same holds with the remainder being in $L^1(H^{-s}(\T)^2,\mu_0^{\otimes 2})$. This shows that the partial path integral \eqref{double_partialPI} with the two fixed boundary values has the same asymptotic behaviour as the full path integral \eqref{fullPI}.

We now turn to the definition of the Sinh-Gordon model. The averages that we have constructed above provide a good candidate premeasure. Recall that $\mathcal{F}$ is a ring that generates the full $\sigma$-algebra, and thus we obtain a unique extension to a non-degenerate measure $\cjg \cdot \cjd_{\mc{C}}$ by Carath\'eodory extension. The fact that the resulting measure is a probability measure follows by approximating the constant function $1$ with functions $F\in \mc{F}$ having support  in $[-N,N]$, letting $N\to \infty$, and using \eqref{bound_EF}.

\begin{definition}[Definition of the Sinh-Gordon model for $R=1$]
Let $\mu >0$, $\gamma \in (0,2)$, $s>0$, and $R=1$. The Sinh-Gordon model on $\mathcal{C}_{1}$ with parameters $\mu$ and $\gamma$ is the unique probability measure $\langle \cdot \rangle_{\mc{C}}$ on $C(\R,H^{-s}(\T))$ with expectation values given by $\langle F\rangle_{\mc{C}}$  for $F \in \mathcal F$.
\end{definition}

\subsection{The general construction and proof of Theorem \ref{thm: sinh and one pt}} \label{subsec: construction sinh R}

We now construct the Sinh-Gordon model on $\mathcal{C}_R$. We begin by defining the approximate probability measures analogously to the case $R=1$. We use the scaling relations for the GFF (see \eqref{scalinrel}) and for the GMC (see Proposition \ref{prop: GMC construction}) to rewrite the path integral on $\mathcal{C}_R$ in terms of a path integral on $\mathcal{C}$.

\begin{definition}
Let $T>0$. Denote by $\langle \cdot \rangle_{\mc{C}_{R,T}}$ the probability measure on $C(\mathbb{R},H^{-s}(\T_R))$ with expectation values given, for any bounded measurable $F \in \mc{F}^R$ with support in $[-T,T]$, by
\begin{equation*}
\langle F \rangle_{\mc{C}_{R,T}}
:=
\frac{1}{Z_{\mc{C}_{R,T}}}\int_\R  \E[F(c+ B_{(\bullet+T) /R} + \varphi_{(\bullet + T)/R}) \, \cdot  e^{- \sum_{\sigma = \pm 1} \mu R^{\gamma Q} e^{\sigma\gamma c} M_\gamma^\sigma(\mc{C}_{[0,2T/R]})}]dc,
\end{equation*}
where $Z_{\mc{C}_{R,T}} = \int_\R \mathbb{E}[e^{- \sum_{\sigma = \pm 1} \mu R^{\gamma Q} e^{\sigma\gamma c} M_\gamma^\sigma(\mc{C}_{[0,2T/R]})}]dc$ is the partition function.
\end{definition}

We may now use our analysis of the model in the case $R=1$ to construct the averages for $R>0$. In particular, the following lemma is a direct consequence of Proposition \ref{prop: sinh expectation}.
The first eigenvalue  $\la_0$ and eigenstate $\psi_0$ of the Hamiltonian ${\bf H}$ on $\mc{C}$ depends on the coupling constant $\mu$. Since 
this coupling constant now plays an important role in passing from the $\mc{C}_R$ case to the $\mc{C}$ case by scaling, 
we shall now denote  $\la_0^\mu$ and $\psi_0^\mu$ for this eigenvalue. 

\begin{lemma} \label{lem: sinh expectation R}
Let $\mu>0$, $\gamma \in (0,2)$, $R>0$, and set $\mu_{R}=\mu R^{\gamma Q}$. Let $F \in \mathcal{F}^R$ with support in $[t_1,t_2]$. Then
\begin{align*}
\lim_{T \rightarrow \infty} \langle F \rangle_{\mc{C}_{R,T}}
=
\langle F \rangle_{\mc{C}_R},	
\end{align*}
where $\langle F \rangle_{\mc{C}_R}\in \mathbb{R}$ is the (deterministic) constant given by the explicit formula
\begin{align}
\langle F \rangle_{\mc{C}_R}
&=	
e^{\lambda_{0}^{\mu_R} \frac{(t_2-t_1)}R} \Big\langle 	\mathbb{E}_{\varphi} \Big[ F(c+B_{(\bullet+t_1)/R}+
\varphi_{(\bullet+t_1)/R})
\\
&\hspace{35mm} \times \psi^{\mu_R}_0(c+B_{(t_2-t_1)/R}+
\varphi_{(t_2-t_1)/R}) e^{- \mu_R \sum_{\sigma = \pm 1} 
e^{\sigma\gamma c} M_\gamma^\sigma(\mc{C}_{[0,(t_2-t_1)/R]})}\big], \psi_0^{\mu_R} \Big \rangle_{\mc{H}}	
\end{align}

\end{lemma}

As for $R=1$, we obtain a unique extension to a probability measure by the Carath\'eodory extension theorem and controlling the normalization. In particular, we may now define the Sinh-Gordon model on $\mathcal{C}_R$.
\begin{definition}[Definition of the Sinh-Gordon model for $R>0$]
Let $\mu >0$, $\gamma \in (0,2)$, and $R>0$. The Sinh-Gordon model on $\mathcal{C}_{R}$ with parameters $\mu$ and $\gamma$ is the unique probability measure $\langle \cdot \rangle_{\mc{C}_R}$ on $C(\R,H^{-s}(\T_R))$ for $s>0$ with expectation values given by $\langle F \rangle_{\mc{C}_R}$ for $F \in \mathcal F^R$.
\end{definition}

We now combine the results of this and the previous subsections to prove Theorem \ref{thm: sinh and one pt}.
\begin{proof}[Proof of Theorem \ref{thm: sinh and one pt}]
The construction of the probability measures corresponding to the Sinh-Gordon models on $\mathcal{C}$ and $\mathcal{C}_R$ in the general case $R>0$ is a consequence of Proposition \ref{prop: sinh expectation} and Lemma \ref{lem: sinh expectation R}, respectively. Moreover, the exact formulas in Theorem \ref{thm: sinh and one pt} correspond to the definition of $\langle F \rangle_{\mc{C}_R}$ for $F \in \mathcal F^R$.
\end{proof}

\subsection{Vertex correlations} \label{subsec:_vertex_correlations}

We now turn to the construction of vertex correlations and the proof of Theorem \ref{thm: sinh correlations}. Recall that an insertion set is a finite set $\mathcal{I} \subset \mathbb{R} \times \mc{C}_R$ and by convention we assume that if $(\alpha,z), (\alpha',z) \in \mathcal{I}$, then necessarily $\alpha = \alpha'$, i.e.\ we do not allow repeated insertions at the same point $z\in \mc{C}_R$. Furthermore, we write $z=(t,\theta)$ for the points on $\mc{C}_R$.

\begin{definition}
Let $\mu > 0$, $\gamma \in (0,2)$, $R>0$, and let $\mathcal{I}$ be an insertion set. For $\eps > 0$, define
\begin{equation} \label{def: vertex prelimit}
\Big\cjg \prod_{(\alpha,z)\in \mc{I}}V^\eps_{\alpha}(z)\Big\cjd_{\mc{C}_R} :=
\Big\langle \prod_{(\alpha, (t, \theta)) \in \mathcal{I}}\eps^{\alpha^2/2}e^{\alpha (c+ B_{t/R} +  \varphi^{R,\eps}(t,\theta)) } \Big\rangle_{\mc{C}_R}\, ,
\end{equation}
where we recall that $\varphi^{R,\eps}$ denotes the circle average process defined by \eqref{circle_average}.
\end{definition}
Without loss of generality, we set $R=1$ and note that the general case $R>0$ may be treated by the scaling relation as in the construction of the measure. For convenience, we assume that the insertion set is of the form (for some $t>0$ fixed)
\begin{equation}\label{defmcI}
\mathcal I
=
\{ (\alpha_{ij}, (t_i, \theta_{ij})) \,|\,  i,j\in [1,N],  0 < t_1 < \dots < t_N<t \}.	
\end{equation}
An insertion set of arbitrary support can be treated by using the translation invariance of the process.

Let $F_\eps(c+ B_\bullet + \varphi_\bullet):=\prod_{(\alpha_{ij}, (t_i, \theta_{ij})) \in \mathcal{I}}\eps^{\alpha_{ij}^2/2}e^{\alpha_{ij}(c+B_{t_i}+\varphi^{\eps}(t_i,\theta_{ij}))}$ (the integrand in the right hand side of \eqref{def: vertex prelimit}).
For any fixed $\beta > 0$, we rewrite the smoothed correlation function as 
\begin{align}
\big\cjg \prod_{(\alpha,z)\in \mc{I}}V^\eps_{\alpha}(z)\big\cjd_{\mc{C}}
&=
\int \E [ G_{\beta,\eps}(c+\varphi) ] e^{-\beta |c|}dc,  
\\
G_{\beta,\eps}(c+\varphi)
&:=
e^{\beta |c|} e^{\lambda_0 t }\mathbb E_\varphi [F_\eps(c+ B_\bullet + \varphi_\bullet) \psi_0(c+B_t+\varphi_t) e^{-\mu\sum_{\sigma = \pm 1} e^{\sigma\gamma c} M^\sigma_\gamma(\mathcal{C}_{[0,t]})}] \psi_0(c+\varphi).	
\end{align}
The parameter $\beta$ is introduced to deduce convergence from uniform integrability of the integrand with respect to a \emph{finite} measure. In addition, we introduce the following notation. For $\eps > 0$ and $\varphi\in H^{-s}(\T)$ the random variable \eqref{def_varphi}
let 
\begin{equation}\label{def_of_h}
h_\eps(\theta)
:=
\sum_{i,j} \alpha_{ij}\mathbb E[ \varphi(\theta) \varphi^\eps_{t_i}(\theta_{ij})], \quad h(\theta): = \sum_{i,j} \alpha_{ij}\mathbb E[ \varphi(\theta) \varphi_{t_i}(\theta_{ij})].
\end{equation}
By \eqref{covvarphiR}, note that $h_\eps$ is  smooth for $\eps>0$ small enough  and converges to $h$ in $C^\infty(\mathbb T)$ as $\eps \rightarrow 0$ (recall that $t_1 > 0$). Furthermore, note that 
\begin{equation}
Ph_\eps(t,\theta) = \sum_{i,j} \alpha_{ij}\mathbb E[ \varphi_t(\theta) \varphi^\eps_{t_i}(\theta_{ij})], \quad Ph(t,\theta) = \sum_{i,j} \alpha_{ij}\mathbb E[ \varphi_t(\theta) \varphi_{t_i}(\theta_{ij})]	
\end{equation}
where $Ph_{\eps}$ is the harmonic function equal to $h_\eps$ at $t=0$ and decaying to $0$ as $t\to +\infty$.

In the following lemma, we will show that the random variables $(G_{\beta,\eps})$ are uniformly integrable with respect to the finite measure $e^{-\beta |c|} \mu_0$ for arbitrary $\beta > 0$.
 This will allow us to take the limit $\eps \rightarrow 0$. In order to determine whether the limit is nontrivial, recall  that we say an insertion set $\mathcal{I}$ is called $\gamma$-admissible if $|\alpha| < Q = \frac{2}{\gamma} + \frac \gamma 2$ for all $(\alpha,z) \in \mathcal{I}$. 

\begin{lemma}\label{lem: vertex correlations}
For every $\beta > 0$, there exists $C>0$ such that
\begin{equation} \label{eq: vc unif l2}
\limsup_{\eps \rightarrow 0}\int \mathbb{E}[G_{\beta,\eps}(c+\varphi)^2] e^{-\beta |c|}dc
\leq
C. 
\end{equation}
Consequently, with the notation \eqref{defmcI}, one has 
\begin{align}\label{eq: vc limit}
\lim_{\eps \rightarrow 0}\Big\cjg \prod_{(\alpha,z)\in \mc{I}}V^\eps_{\alpha}(z)\Big\cjd_{\mc{C}}	
&=
e^{\lambda_0 t+\frac{1}{2}\sum_{i,j}\alpha_{ij}^2t_i}  \int_{\R} e^{\sum_{i,j}\alpha_{ij} c}\mathbb{E}\big[\psi_0\big(c+B_t+\sum_{i,j} \alpha_{ij}t_i+ \varphi_t + Ph (t,\cdot))\psi_0(c+\varphi+h) 
\\
& \qquad \qquad \times e^{-\mu \sum_{\sigma = \pm 1} e^{\sigma\gamma c} \int_{\mathcal{C}_{[0,t]}} \prod_{i,j} |e^{-s+i\theta}-e^{-t_i+i\theta_{ij}}|^{-\gamma \sigma \alpha_{ij}} M^\sigma_\gamma(dsd\theta)}\big ]dc
\end{align}
and the limiting correlation is nontrivial if and only if $\mathcal I$ is $\gamma$-admissible.

\end{lemma}

\begin{proof}
We apply the Girsanov transformation to remove the terms $e^{\alpha_{ij}(B_{t_i}+\varphi^\eps(t_i,\theta_{ij}))}$ 
and get a shift of the field by $\sum_{ij} \alpha_{ij} t_i+Ph_\eps$: this yields
\begin{align*}
& \mathbb{E}[G_{\beta,\eps}(c,\varphi)^2]\\
& =
e^{2\beta|c|+2\lambda_0 t+\frac{1}{2}\sum_{i,j}\alpha_{ij}^2t_i+2\sum_{i,j} \alpha_{ij} c } \mathbb E [ \psi_0(c+B_t + \sum_{i,j} \alpha_{ij} t_i+ \varphi_t + Ph_\eps(t,\cdot))^2 \psi_0(c+\varphi+h_\eps)^2 e^{-\mu \, \mathcal{U}_{t,\eps} (c,\varphi)}],
\end{align*}
where 
\begin{equation}
\mathcal{U}_{t,\eps}(c,\varphi)
:=
\sum_{\sigma = \pm 1} e^{\sigma\gamma c} \int_{\mathcal{C}_{[0,t]}} \int_{\mathbb{T}} \prod_{ij} |e^{-s+i\theta}-e^{-t_i^{\eps}(v)+i\theta_{ij}^{\eps}(v)}|^{-\gamma \sigma \alpha_{ij}} \frac{dv}{2\pi} M^\sigma_\gamma(dsd\theta),
\end{equation}
and $(t_i^\eps(v), \theta_{ij}^\eps(v)) = (t_i,\theta_{ij}) + \eps(\cos(v),\sin(v)) $.
We bound $e^{-\mu \, \mathcal{U}_{t,\eps} (c,\varphi)}$ by $1$ and use the Cauchy-Schwarz inequality to get
 \[\begin{split}
 \mathbb{E}[G_{\beta,\eps}(c+\varphi)^2]
\leq & e^{2\beta|c|+2\lambda_0 t+\frac{1}{2}\sum_{i,j}\alpha_{ij}^2t_i+2\alpha_{ij} c }  \E [ \psi_0(c+B_t + \sum_{i,j} \alpha_{ij} t_i+ \varphi_t + Ph_\eps(t,\cdot))^4]^{1/2} \\
& \times \E[\psi_0(c+\varphi+h_\eps)^4]^{1/2}. 
\end{split}\]

Note that since $t>0$, $Ph(t,\cdot):=\lim_{\eps \rightarrow 0} Ph_\eps(t,\cdot)$ is smooth. Hence, by the Cameron-Martin theorem we have that there exists $C>0$ such that uniformly in $\eps$ sufficiently small
\begin{equation} \label{eq: vc 1}
\mathbb{E}[\psi_0(c+B_t + \sum_{i,j} \alpha_{ij} t_i+ \varphi_t + Ph_\eps(t,\cdot))^4]
\leq
C \,  \mathbb{E}[\psi_0(c+B_t+ \varphi_t)^4] \leq C\E [e^{-t\H_0}\psi_0^4]. 	
\end{equation}
Similarly, since $h_\eps$ is uniformly smooth as $\eps\to 0$, by the Cameron-Martin theorem
\[ \E[\psi_0(c+\varphi+h_\eps)^4]  \leq C \E[\psi_0(c+\varphi)^4].\]

Combining these estimates, Lemma \ref{lem:eigenfunction_decay} and the boundedness of $e^{-t{\bf H}^0}$ on 
$L^1(H^{-s}(\T),\mu_0)$ (Proposition \ref{prop:FreeHamiltonian})
there exists $C,C'>0$, depending on $t$  such that uniformly in $\eps$,
\begin{equation}
\int \mathbb E [G_{\beta,\eps}^2(c+\varphi)] e^{-\beta|c|}dc
\leq C \int (e^{\beta|c|+\sum_{i,j}\alpha_{ij} c} \, \mathbb{E}[\psi_0^{4}]+ \E[e^{-t{\bf H}^0}(\psi_0^{4})]) dc\leq C'.
\end{equation}
This establishes the uniform $L^2$ estimate \eqref{eq: vc unif l2}.

We now turn to \eqref{eq: vc limit}. Note that $G_{\beta,\eps}$ converges pointwise to $e^{|\beta| c}$ times the integrand in \eqref{eq: vc limit}. Furthermore, $\eqref{eq: vc unif l2}$ implies that $(G_{\beta,\eps})$ is uniformly integrable with respect to the finite measure $e^{-\beta|c|}\mu_0$.  Hence, by Vitali's convergence theorem, we obtain convergence of the corresponding integrals. Rearranging the factor $e^{-\beta|c|}$ yields \eqref{eq: vc limit}. To prove nontriviality, note that all the expressions in \eqref{eq: vc limit} are positive almost everywhere \emph{except} for the term involving the GMC. Thus it is sufficient to establish that
\begin{equation}
e^{\sum_{i,j} \alpha_{ij} c}\exp \big(-\mu \sum_{\sigma = \pm 1} e^{\sigma\gamma c} \int_{\mathcal{C}_{[0,t]}} \prod_{i,j} |e^{-s+i\theta}-e^{-t_i+i\theta_{ij}}|^{-\gamma \sigma \alpha_{ij}} M^\sigma_\gamma(dsd\theta)\big) > 0 \, \text{ a.e., }	
\end{equation}
which is in turn true (i.e.\ the term in the exponential is finite almost surely) if and only if $\mathcal I$ is $\gamma$-admissible, see \cite[Lemma 5.19]{Berestycki_lqggff} and in particular Steps 2 and 3 there.
\end{proof}

We now turn to the proof of Theorem \ref{thm: sinh correlations}.
\begin{proof}[Proof of Theorem \ref{thm: sinh correlations}]
For $R=1$, the convergence, nontriviality, and explicit formula of vertex correlations with arbitrary insertion set follows directly from Lemma \ref{lem: vertex correlations}. The generic case $R>0$, as we have argued previously, follows from the scaling relation for GMC. We observe that for the $1$-point function, if the insertion is at $(t_1,\theta_1)=(0,0)$ with weight $|\alpha|<Q$, we can let $t\to 0$ in the expression \eqref{eq: vc limit} and obtain (with $h(\theta)=-\alpha \log|e^{i\theta}-1|$ as in \eqref{def_of_h}):
\begin{equation}\label{1-point}
\cjg V_{\alpha}(0)\cjd_{\mc{C}}=  \|e^{\alpha c/2} \psi_0(\cdot + h)\|^2_{\mc{H}}.
\end{equation}

We now turn to the exponential decay of correlations for the truncated two-point function, or covariance, of the vertex correlations. We will fix $R=1$ for ease of notation; the case $R>0$ follows with straightforward modifications. As a preliminary step, we analyze covariances of bounded functions of compact support. Let $\eps > 0$. Using the same approach as in the proof of Proposition \ref{prop: sinh expectation}, for every $F_\eps \in \mathcal A$ bounded and supported on $[-\eps,\eps]$, define the map
\begin{equation}
\mc{B}_{F_\eps}:\mathcal H \rightarrow \mathcal H, \quad (\mc{B}_{F_\eps} \psi)(c,\varphi) = \mathbb E_\varphi[F_\eps(c+B_{\bullet +\eps} +\varphi_{\bullet +\eps}) \psi(c+B_{2\eps}+\varphi_{2\eps})e^{-\mu  \sum_{\sigma = \pm 1} e^{\sigma\gamma c}M_\gamma^\sigma([0,2\eps]\times\T)}].
\end{equation}
In particular, when $\eps = 0$ we have that $\mc{B}_{F_0}\psi = F_0 \psi  \in \mathcal {H}$. Given $F_\eps,G_\eps \in \mathcal A$ that are bounded and supported on $[-\eps,\eps]$, by the domain Markov property and tower property of conditional expectations, we have
\begin{equation}
\begin{split}
\langle F_\eps \cdot \tau_{t} G_\eps \rangle_{\mathcal C_1} =
 e^{\lambda_0 (t+2\eps)} \cjg \mc{B}_{F_\eps}e^{-(t-2\eps){\bf H}}\mc{B}_{G_\eps}\psi_0,\psi_0\cjd_{\mc{H}},
\end{split}
\end{equation}
where we recall that $\tau_t G$ translates the support of $G$ by $+t$. 

It is straightforward to check that the operator $\mc{B}_{F_\eps}$ is self-adjoint (using that $F_\eps$ is real valued and the 
dynamics $B_s+\varphi_s$ is reversible with respect to $\mu_0$), thus the eigenfunction expansion in Theorem \ref{thm:spectrum_H} yields (for $t>2\eps)$: 
\[\begin{split}
\langle F_\eps \cdot \tau_{t} G_\eps \rangle_{\mathcal C_1} &= e^{\lambda_0(t+2\eps)} \sum_{j \geq 0} e^{-\lambda_j (t-2\eps)} \langle  \mc{B}_{F_\eps}\psi_0,\psi_j \rangle_{\mathcal H} \langle \mc{B}_{G_\eps}\psi_0,\psi_j\rangle_{\mathcal H}.
\end{split}\]
Hence since $\langle F_\eps \rangle_{\mathcal C_1}=e^{2\la_0\eps}\cjg \mc{B}_{F_\eps}\psi_0,\psi_0\cjd_{\mc{H}}$ and $\langle G_\eps \rangle_{\mathcal C_1}=e^{2\la_0\eps}\cjg \mc{B}_{G_\eps}\psi_0,\psi_0\cjd_{\mc{H}}$,
we have the representation
\begin{equation} \label{eq: covariances}
{\rm Cov}(F,\tau_t G) := \langle F_\eps \cdot \tau_tG_\eps \rangle_{\mathcal C_1} - \langle F_\eps \rangle_{\mathcal C_1} \langle \tau_t G_\eps \rangle_{\mathcal C_1} = e^{-(\lambda_1-\lambda_0)t} \mathcal R_t(F_\eps, G_\eps),
\end{equation}
where
\begin{equation}
\mathcal R_t(F_\eps, G_\eps):= e^{2(\lambda_1+\lambda_0)\eps} \sum_{j \geq 1} e^{-(\lambda_j - \lambda_1)(t-2\eps)} \langle  \mc{B}_{F_\eps}\psi_0,\psi_j \rangle_{\mathcal H} \langle \mc{B}_{G_\eps}\psi_0,\psi_j\rangle_{\mathcal H} 
\end{equation}
In particular, when $F_\eps = G_\eps$, the covariance is non-negative.

Let us now turn to vertex correlations. Fix $t>0$ without loss of generality. Let $|\alpha_1|,|\alpha_2|<Q$ and $\theta_1,\theta_2 \in \mathbb T$. For each $K \in \mathbb N^*$, write
\begin{equation}
V_{\alpha}^\eps(s,\theta)_K:= \min(V_{\alpha}^\eps(s,\theta),K) \in \mathcal A,
\end{equation}
which is bounded and supported on $[-\eps, \eps]$. Hence, by \eqref{eq: covariances} we obtain that for every $\eps > 0$ such that $2\eps < t$ and for every $\delta > 0$,
\begin{equation}
{\rm Cov}\Big(V_{\alpha_1}^\eps(0,\theta_1)_K,  V_{\alpha_2}^\eps(t,\theta_2)_K  \Big) = e^{-(\lambda_1-\lambda_0)t} e^{(\lambda_1-\lambda_0)\delta} e^{-(\lambda_1-\lambda_0)\delta }\mathcal R_t(V_{\alpha_1}^\eps(0,\theta_1)_K,  V_{\alpha_2}^\eps(0,\theta_2)_K).
\end{equation}
We now choose $\eps$ and $\delta$ such that $2\eps < \delta < t$. Then by the Cauchy-Schwarz inequality, 
\begin{equation}
|e^{-(\lambda_1-\lambda_0)\delta }\mathcal R_t(V_{\alpha_1}^\eps(0,\theta_1)_K,  V_{\alpha_2}^\eps(0,\theta_2)_K)| \leq e^{-(\lambda_1-\lambda_0)\delta }\prod_{i=1,2} \mathcal R_\delta( 	V_{\alpha_i}^\eps(0,\theta_i)_K, V_{\alpha_i}^\eps(0,\theta_i)_K)^{1/2}.
\end{equation}
However, by \eqref{eq: covariances} and the fact that the covariances are positive when $F_\eps=G_\eps$, we obtain 
\begin{align*}
& \limsup_{\eps \rightarrow 0}\limsup_{K \rightarrow \infty}\Big|\prod_{i=1,2} \mathcal R_\delta( 	V_{\alpha_i}^\eps(0,\theta_i)_K, V_{\alpha_i}^\eps(0,\theta_i)_K)^{1/2}\Big|e^{-(\lambda_1-\lambda_0)\delta }	\\
& \leq \lim_{\eps \rightarrow 0} \lim_{K \rightarrow \infty} \prod_{i=1,2} {\rm Cov}\Big(V_{\alpha_i}^\eps(0,\theta_i)_K, V_{\alpha_i}^\eps(\delta,\theta_i)_K  \Big)^{1/2}.
\end{align*}
But, since the covariances of vertex correlations converge, we therefore have that
\begin{equation}
\Big|{\rm Cov}\Big(V_{\alpha_1}(0,\theta_1), V_{\alpha_2}(t,\theta_2)\Big)\Big| \leq e^{-(\lambda_1-\lambda_0)t} C(\alpha_1,\theta_1; \alpha_2,\theta_2),	
\end{equation}
where
\begin{equation}
C(\alpha_1,\theta_1;\alpha_2,\theta_2):= e^{(\lambda_1-\lambda_0)\delta} \prod_{i=1,2}{\rm Cov}\Big(V_{\alpha_i}(0,\theta_i), V_{\alpha_i}(\delta,\theta_i) \Big)^{1/2}. 	\qedhere
\end{equation}
\end{proof}

\appendix

\section{Heuristics for the path integral} \label{appendix: path integral}

\subsection{Brownian motion and Ornstein-Uhlenbeck processes from the path integral}\label{sec:Brownian_OU}

The formal one-dimensional path integral over the space $C^0_x([0,T]):=\{\phi \in C^0([0,T],\R)\,|\, \phi(0)=x\}$ of continuous functionswith $T>0$ and $x\in \R$
\begin{equation}\label{1dPathIntBM}
\int_{C^0_x([0,T])} F(\phi) e^{-\frac{1}{2 \sigma^2}\int_{0}^T |\pl_t\phi(t)|^2dt} D\phi
\end{equation}
can be rigorously defined using the Brownian motion $B^\sigma_t$ with initial condition $B^\sigma_0=0$ and covariance 
$\E[B_t^\sigma B^\sigma_s]=\sigma^2\min(s,t)$ by
\[\int_{C^0_x([0,T])} F(\phi) e^{-\frac{1}{2 \sigma^2}\int_{0}^T |\pl_t\phi(t)|^2dt} D\phi:=\E[F(x+B^\sigma)].\]
Notice that $B^\sigma_t=B_{\sigma^2t}$ in law, with $B_t=B^1_t$.
The heuristic consists in first taking $x=0$, in which case  one looks for a Gaussian random variable $B$ (formally in $H^1$) such that $\E[\cjg f,B\cjd_{H^1}\cjg g,B\cjd_{H^1}]=\cjg f,g\cjd_{H^1}$ for all 
$f,g\in H^1_0([0,T]):=\{f\in H^1([0,T])\,|\, f(0)=0\}$ with $\cjg f,g\cjd_{H^1}=\int_0^T\pl_tf(t)\pl_tg(t)dt$. A direct calculation 
gives that $B$ must satisfy $\E[B_sB_t]=\min(s,t)$  and we thus set $B$ to be the Brownian motion (the drawback is that $B$ is only $C^{1/2-\eps}$ H\"older). 
Then we add $x$ to $B^\sigma$ to represent the field with values $x$ at $t=0$.

We can apply the same kind of reasoning for the formal 1 dimensional path integral over $C^0_x([0,T])$ 
\begin{equation}\label{1dPathInt}
\int_{C^0_x([0,T])} F(\phi) e^{-\frac{1}{2 \sigma^2}\int_{0}^T( |\pl_t\phi(t)|^2+(\mu\sigma)^2|\phi (t)|^2) dt } D\phi
\end{equation}
for $\mu,\sigma>0$. This can be rigorously defined using the Ornstein-Uhlenbeck process $X_t$, which satisfies the following stochastic differential equation
\begin{equation}\label{OUSDE}
dX_t=-\mu \sigma X_t dt+ \sigma dB_t.
\end{equation}
This is a diffusion with generator given on $\R$ by
\begin{equation*}
\mathcal{L} f(s)=-\frac{\sigma^2}{2} f''(s)+ \mu \sigma s f'(s).
\end{equation*}
Starting from a point $X_0=x$ one can explicitly solve \eqref{OUSDE} yielding the formula
\begin{equation}\label{OUsolution}
X_t=x e^{-\mu \sigma t}+ \sigma e^{-\mu \sigma t} \int_0^t e^{ \mu \sigma s} dB_s= x e^{-\mu \sigma t}+ \tilde{X}_t
\end{equation} 
where $\tilde{X}_t$ is a Gaussian random variable with covariance structure
\begin{equation*}
\E[\tilde{X}_t \tilde{X}_s]= \frac{\sigma}{2 \mu} (e^{-\mu \sigma |t-s|}-e^{-\mu \sigma (t+s)} ).
\end{equation*}
Hence at stationarity the covariance is given by $\E[X_s X_t]= \frac{\sigma}{2\mu} e^{-\mu \sigma |t-s|}$.
The path integral \eqref{1dPathInt} can be given a sense as follows:
\[  \int_{C^0_x([0,T])} F(\phi) e^{-\frac{1}{2 \sigma^2}\int_{0}^T 
|\pl_t\phi(t)|^2+(\mu\sigma)^2 |\phi (t)|^2 dt } D\phi:= \E[F(X)]\]

\subsection{The GFF path integral}

We now analyse the path integral \eqref{eq: formal def approx sinh} underlying our construction and try to identify the associated Markov process as in the 1d case. Recall that a real valued smooth function $\phi$ on the cylinder $\mc{C}_R=\R\times \T_R$ can be decomposed using Fourier series in $\theta$:
\begin{equation}
\phi(t, \theta)= \phi_0(t)+ \sum_{n \not = 0} \phi_{n}(t) e^{i n \frac{\theta}{R}}. 
\end{equation}
The same also holds on $\mc{C}^+_{R,T}=[0,T]\times \T_R$. 
With this decomposition, and writing $\phi_{n}(t)$ under the form $\phi^R_{n}(t) = \frac{x_n(t)+i y_n(t)}{2 \sqrt{n}}$ with $x_n(t),y_n(t)$ real valued, a direct calculation gives
\begin{align*}
\frac {1} {4 \pi} \int_{\mc{C}^+_{R,T}}  |\nabla \phi|^2 dtd\theta & =   \frac {1} {4 \pi}   \int_{\mc{C}^+_{R,T}}  |\pl_t\phi_0(t)|^2 + \sum_{n \not = 0} |\pl_t \phi_{n}(t)|^2 dt d\theta +  \frac {1} {4 \pi R^2}  \int_{\mc{C}^+_{R,T}}   \sum_{n \not = 0}  |n \phi_n(t)|^2 dt d\theta \\
& = \frac {R} {2}   \int_{0}^T  |\pl_t\phi_0(t)|^2 dt +  \sum_{n \geq 1} \left ( R  \int_{0}^T   |\pl_t\phi_n(t)|^2 dt  
+  \frac{n^2}{R} \int_{0}^T  |\phi_n(t)|^2 dt \right )\\
& =  \frac {R} {2}   \int_{0}^T  |\pl_t\phi_0|^2 dt +\sum_{n \geq 1} \left (\frac{R}{4 n}  \int_{0}^T   |\pl_tx_n(t)  |^2 dt  +    \frac{n}{4R} \int_{0}^T     |  x_n(t) |^2 dt \right )  \\
& +\sum_{n \geq 1} \left ( \frac {R} {4n} \int_{0}^T   |\pl_t y_n(t)|^2 dt  +  \frac{ n }{4R}  \int_{0}^T  |  y_n(t) |^2 dt \right ).
\end{align*}
Let us consider the formal Gaussian path integral 
\[ \int F(\phi)e^{-\frac{1}{4\pi} \int_{\mc{C}^+_{R,T}}  |\nabla \phi|^2dtd\theta}d\phi\]
on the cylinder with the condition that 
\[\phi(0,\theta)=c+\sum_{n\in \Z\setminus \{0\}}\varphi_ne^{in\frac{\theta}{R}}, \quad 
\varphi_n:=\left\{\begin{array}{ll}
\frac{x_n+iy_n}{2\sqrt{n}} & n>0,\\
\frac{x_{|n|}-iy_{|n|}}{2\sqrt{|n|}} & n<0,
\end{array}\right.\]
for some real $x_n,y_n$. Then, by the discussion above,
the field $\phi$ can be represented under the form 
\[ \phi^R(t,\theta)=\phi_0^R(t)+\sum_{n\in \Z\setminus \{0\}}\varphi_n^R(t)e^{in\frac{\theta}{R}},\] 
where 
\[\phi_0^R(t):=c+B_{t/R}, \quad  \varphi^R_n(t):=\Bigg\{\begin{array}{ll}
\frac{x_n^R(t)+iy_n^R(t)}{2\sqrt{n}} & n>0,\\
\frac{x_{|n|}^{R}(t)-iy^R_{|n|}(t)}{2\sqrt{|n|}} & n<0,
\end{array}\]
with $x_n^R(t),y_n^R(t)$ the Ornstein-Uhlenbeck processes with parameters $\mu_n= \sqrt{\frac{n}{2R}}$ and $\sigma_n=\sqrt{\frac{2n}{R}}$.

We now describe how to interpret the path integral on the forward cylinder $\mathcal{C}_{R,T}^+$ (we will relate it to the path integral on $\mathcal{C}_{R,T}$ by translations below). For $(c,\varphi)\in \R\times H_0^{-s}(\T_R)$ fixed with $\varphi=\sum_{n>0}\frac{(x_n+iy_n)}{2\sqrt{n}}e^{i\frac{n}{R}\theta}+\sum_{n>0}\frac{(x_n-iy_n)}{2\sqrt{n}}e^{-i\frac{n}{R}\theta}$, 
 let us define the following space of distributions on the cylinder with fixed values $(c,\varphi)\in \R\times H_0^{-s}(\T_R)$ at $t=0$ for some fixed $s>0$:
 \[\begin{split}
  E_{c,\varphi}(\mc{C}^+_{R,T})=\Big\{ \phi(t,\theta)=& x_0(t)+ \sum_{n\not=0} \frac{x_{|n|}(t)+i\,{\rm sign}(n)y_{|n|}(t)}{2\sqrt{|n|}}e^{in\frac{\theta}{R}} \in H^{-s}(\mc{C}^+_{R,T})\\
 &  \quad \,\Big|\, x_0(t)\in C^0_{c}([0,T]), \forall n\geq 1,\,  x_n(t)\in C^0_{x_n}([0,T]), y_n(t)\in C^0_{y_n}([0,T])\Big\}.
 \end{split}\]
 The formal Gaussian measure on the cylinder $\mc{C}^+_{R,T}$ is defined for bounded measurable $F$ by 
 \[ \int_{E_{c,\varphi}(\mc{C}^+_{R,T})}F(\phi)e^{-\frac{1}{4\pi}\int_{\mc{C}^+_{R,T}}|\nabla \phi|^2dtd\theta}D\phi:= \E_\varphi[ F( \phi^R(t,\theta))],\]
 where  $\phi^R(t,\theta)$ is the random variable in $E_{c,\varphi}(\mc{C}^+_{R,T})$ defined by
 \[ \phi^R(t,\theta):=c+B_{\frac{t}{R}}+\sum_{n=1}^\infty \frac{x_n(\frac{t}{R})+iy_n(\frac{t}{R})}{2\sqrt{n}}e^{in\frac{\theta}{R}}+ 
 \frac{x_n(\frac{t}{R})-iy_n(\frac{t}{R})}{2\sqrt{n}}e^{-in\frac{\theta}{R}},\]
with $x_n(t),y_n(t)$ the Ornstein-Uhlenbeck processes defined above for $R=1$, and the expectation is conditional on $\varphi$.

 Now, to work on the strip $\mc{C}_{R,T}=[-T,T]\times \T_R$, we view it as a shift by $t\mapsto t-T$ of $\mc{C}^+_{R,2T}$ and therefore represent the path integral on $\mc{C}_{R,T}$ as 
 \[\int_{E_{c,\varphi}(\mc{C}_{R,T})}F(\phi)e^{-\frac{1}{4\pi}\int_{\mc{C}_{R,T}}|\nabla \phi|^2dtd\theta}D\phi:= 
 \E[ F( \phi^R(t,\theta))],\]
 where  $E_{c,\varphi}(\mc{C}_{R,T})=\{ \phi(t+T,\theta) \,|\, \phi \in  E_{c,\varphi}(\mc{C}^+_{R,2T})\}$ and now the random variable
 $\phi^R$ 
 used to define the path integral is 
 \[ \phi^R(t,\theta)=c+B_{\frac{t+T}{R}}+\sum_{n=1}^\infty \frac{x_n(\frac{t+T}{R})+iy_n(\frac{t+T}{R})}{2\sqrt{n}}e^{in\frac{\theta}{R}}+\sum_{n=1}^\infty \frac{x_n(\frac{t+T}{R})-iy_n(\frac{t+T}{R})}{2\sqrt{n}}e^{-in\frac{\theta}{R}}.\]

\bibliographystyle{alpha}
\bibliography{Sinh}
\end{document}